\documentclass[11pt]{amsart}
\usepackage[utf8]{inputenc}

\usepackage{amsmath,amsthm}
\usepackage{amssymb}
\usepackage{epsfig,tikz}
\usetikzlibrary{cd}
\usepackage{graphicx}
\usepackage{asymptote}
\usepackage{xy}
\usepackage{xypic}

\usepackage[tmargin=1in,bmargin=1in,lmargin=1in,rmargin=1in]{geometry}

\usepackage{todonotes}

\newcommand{\Mbar}{\overline{M}}
\newcommand{\bdpt}{\beta}

\newcommand{\bA}{\mathbb{A}}

\newcommand{\bP}{\mathbb{P}}
\newcommand{\PP}{\mathbb{P}}

\newcommand{\beq}{\begin{equation}}
\newcommand{\eeq}{\end{equation}}

\newcommand{\PF}{\mathrm{PF}}
\newcommand{\CPF}{\mathrm{CPF}}

\newcommand{\injto}{\hookrightarrow}

\newcommand{\defn}{\textbf}

\newcommand{\emb}{\Phi}
\newcommand{\kapn}{|\psi_n|}

\newcommand{\Tour}{\mathrm{Tour}}

\newcommand{\forget}{\pi_{\mathrm{lazy}}}

\newtheorem{thm}{Theorem}
\newtheorem{lemma}[thm]{Lemma}
\newtheorem{prop}[thm]{Proposition}
\newtheorem{corollary}[thm]{Corollary}

\numberwithin{thm}{section}
\numberwithin{equation}{section}
\numberwithin{figure}{section}

\theoremstyle{definition}

\newtheorem{example}[thm]{Example}
\newtheorem{definition}[thm]{Definition}

\newtheorem{remark}[thm]{Remark}

\newtheorem*{HyperplanesThm}{Theorem \ref{thm:simple-hyperplanes}}
\newtheorem*{TournamentsThm}{Theorem \ref{thm:tournaments}}
\newtheorem*{TotalCor}{Corollary \ref{cor:total-degree}}

\usepackage{mathtools}
\usepackage{enumitem}

\def\multichoose#1#2{\left<\genfrac{}{}{0pt}{}{#1}{#2}\right>}

\title{Lazy tournaments and multidegrees of a projective embedding of $\Mbar_{0,n}$}
\author{Maria Gillespie}
\thanks{Maria Gillespie was partially supported by NSF DMS award number 2054391.}
\address{Department of Mathematics, Colorado State University, Fort Collins, CO, USA} \email{maria.gillespie@colostate.edu} 

\author{Sean T. Griffin}
\thanks{Sean T. Griffin was partially supported by NSF Grant DMS-1439786 while in residence at the Institute for Computational and Experimental Research in Mathematics in Providence, RI, during the Spring 2021 semester.}
\address{Department of Mathematics, University of California Davis, Davis, CA, USA}
\email{stgriffin@ucdavis.edu}

\author{Jake Levinson}
\thanks{Jake Levinson was partially supported by an AMS Simons Travel Grant and by NSERC Discovery Grant RGPIN-2021-04169.}
\address{Department of Mathematics, Simon Fraser University, Burnaby, BC, Canada}
\email{jake\_levinson@sfu.ca}

\date{\today}

\begin{document}

\maketitle{}

\begin{abstract}
     We provide a new geometric interpretation of the multidegrees of the (iterated) Kapranov embedding $\emb_n:\Mbar_{0,n+3}\injto \PP^1\times \PP^2\times \cdots \times \PP^n$, where $\Mbar_{0,n+3}$ is the moduli space of stable genus $0$ curves with $n+3$ marked points. We enumerate the multidegrees by disjoint sets of boundary points of $\Mbar_{0,n+3}$ via a combinatorial algorithm on trivalent trees that we call a \textit{lazy tournament}.  These sets are compatible with the forgetting maps used to derive the recursion for the multidegrees proven in 2020 by Gillespie, Cavalieri, and Monin.
 
 The lazy tournament points are easily seen to total $(2n-1)!!=(2n-1)\cdot (2n-3) \cdots 5 \cdot 3 \cdot 1$, giving a natural proof of the fact that the total degree of $\emb_n$ is the odd double factorial. This fact was first proven using an insertion algorithm on certain parking functions, and we additionally give a bijection to those parking functions.
\end{abstract}

\section{Introduction}

In this paper, we give a new interpretation of the \emph{multidegrees} of the Deligne--Mumford moduli space $\Mbar_{0,n+3}$ \cite{deligne-mumford1969} of genus-$0$ stable curves with $n$ marked points, under the projective embedding
\[\Phi_n : \Mbar_{0, n+3} \hookrightarrow \mathbb{P}^1 \times \cdots \times \mathbb{P}^n\]
called the \emph{iterated Kapranov map}.
This map, first studied by Keel and Tevelev \cite{KeT}, is one of the simplest ways to realize $\Mbar_{0, n+3}$ as a projective variety. The divisor classes associated to it are the \emph{omega classes}, modifications of the better-known \emph{psi classes}. The multidegrees count the intersection points of $\Mbar_{0,n+3}$ with specified numbers of general hyperplanes pulled back from each $\mathbb{P}^i$ factor, i.e., they correspond to intersection products of omega classes in the cohomology ring of $\Mbar_{0,n+3}$.

Prior work \cite{CGM} showed that the multidegrees can be enumerated by certain parking functions. It was
also shown that the \emph{total degree} of $\emb_n$ is the odd double factorial $(2n-1)!!$.  The proof relied on developing an intricate insertion algorithm on column-restricted parking functions, without a clear connection to geometry.  However, the quantity $(2n-1)!!$ suggests such a connection, because it is also the total number of trivalent trees on $n+2$ (not $n+3$) leaves, i.e., the total number of boundary points on $\Mbar_{0, n+2}$.

We bridge this gap in this paper by associating to each multidegree a set of boundary points (labeled trivalent trees) on $\Mbar_{0,n+3}$. To do this, we develop an algorithm we call a \emph{lazy tournament} that produces these boundary points; see Definition \ref{def:lazy-tournament}. We show that these points are compatible with the forgetting and relabeling maps used in the (asymmetric) \emph{string equation} \cite{CGM} governing the multidegrees; see Proposition \ref{prop:recursion} and Remark \ref{rmk:forget-and-forgive}. The lazy tournament points then give an immediate count of the total degree of $\emb_n$: they partition the complete set of boundary points (trees) on a stratum of $\Mbar_{0,n+3}$ isomorphic to $\Mbar_{0,n+2}$. As such, they visibly total $(2n-1)!!$; see Theorem \ref{thm:tournaments} and Corollary \ref{cor:total-degree}.  Finally, we give a direct bijection between the column-restricted parking functions of \cite{CGM} and the lazy tournament points.

These results add to a growing body of literature relating algebraic combinatorics and the geometry of moduli spaces of curves. The basic connection to trees via the boundary stratification (see Section \ref{sec:background}) is long established. Enumerative questions have been of particular interest recently, including examining (as in this paper) many intersection products and structure constants on $\Mbar_{g,n}$ \cite{canning2021chow, silversmith2021crossratio}, tautological relations \cite{CladerJanda, pandharipande2020relations, PixtonThesis}, and Schubert calculus involving limit linear series \cite{chan-pflueger, eisenbud-harris-limit-linear}. Other topics of interest include the $S_n$ action on $H^*(\Mbar_{0,n})$ over $\mathbb{C}$ \cite{BergstromMinabe, Getzler, RaSil2020} and $\mathbb{R}$ \cite{Rains}, Chern classes of vector bundles on $\Mbar_{0,n}$ associated to $\mathfrak{sl}_r$ \cite{damiolini2020vertex, gibneykeelmorrison2002}, explicit projective equations for $\Mbar_{0,n}$ \cite{MonRan}, and similar questions pertaining to a number of closely-related moduli spaces \cite{clader2021permutohedral, clader2020boundary, fry2019tropical, larson2020global, sharma2019intersections}.

\subsection{Lazy tournaments and multidegrees}

The \defn{multidegrees} of the embedding $\emb_n$ are defined as 
\begin{equation}\label{eq:multidegree}
    \deg_{(k_1,\ldots,k_n)}(\emb_n)=\int_{\PP^1\times \cdots \times \PP^n}[\emb_n(\Mbar_{0,n+3})]\prod H_i^{k_i},
\end{equation} where $H_i$ is the class in the Chow ring of $\PP^1\times \cdots \times \PP^n$ pulled back from a hyperplane in $\PP^i$.  That is, the multidegree $\deg_{(k_1,\ldots,k_n)}(\emb_n)$  is the expected number of points of the intersection of $\emb_n(\Mbar_{0,n+3})$ with $n$ generic hyperplanes, $k_i$ of which are from $\PP^i$ for all $i$.  The \defn{total degree} of the embedding $\emb_n$ is the sum of all the multidegrees, defined as  $$\deg(\emb_n)=\sum_{k_1+k_2+\cdots+k_n=n} \deg_{(k_1,\ldots,k_n)}(\emb_n).$$ Note that the total degree is the (ordinary) degree of the projectivization of the affine cone over $\emb_n(\Mbar_{0,n+3})$ formed by lifting to the affine space $\bA^2\times \bA^3 \times \cdots \times \bA^{n+1}=\bA^{n(n+3)/2}$ lying above $\PP^1\times \cdots \times \PP^n$ (see Van der Waarden \cite{van}).

In \cite{CGM}, Cavalieri, the first author, and Monin showed the following.
\begin{thm}[\cite{CGM}]\label{thm:intro-odd-double-factorial}
 The total degree of $\emb_n$ is \begin{equation}\label{eq:total-degree}\deg(\emb_n)=(2n-1)!!=(2n-1)\cdot (2n-3)\cdot \cdots \cdot 5\cdot 3 \cdot 1.\end{equation}  Moreover, letting $\CPF(k_1,\ldots,k_n)$ be the number of \textbf{column-restricted parking functions} (see Section \ref{sec:CPFs}) having column heights $k_1,\ldots,k_n$, we have $$\deg_{(k_1,\ldots,k_n)}(\emb_n)=|\CPF(k_1,\ldots,k_n)|,$$ and the total number of column restricted parking functions of height $n$ is $(2n-1)!!$.
\end{thm} 
The first main result of this paper is an alternative interpretation of the multidegrees in terms of the boundary points of $\Mbar_{0,n+3}$ lying on a stratum isomorphic to $\Mbar_{0,n+2}$. To state the result precisely, we first recall that a boundary point of $\Mbar_{0,n+3}$ may be represented as a leaf-labeled \defn{trivalent tree}, that is, a tree with $n+3$ labeled leaves for which every vertex has degree $1$ or $3$.  As a convenient choice of notation, we label the $n+3$ leaves with $a,b,c,1,2,3,\ldots,n$, as in Figure \ref{fig:example-tournament}, and assume the ordering $a<b<c<1<2<\ldots<n$.  We sometimes use the notation $\Mbar_{0,\{a,b,c,1,2,\ldots,n\}}$, or in general $\Mbar_{0,X}$, to indicate that we are labeling our marked points by a set of labels $X$.

\begin{figure}
    \centering
    \includegraphics{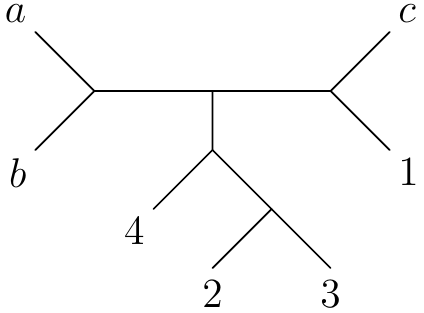}\hspace{1.5cm} \includegraphics{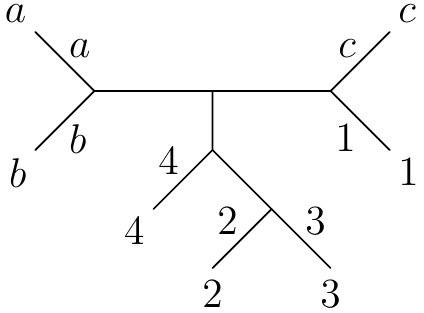}\hspace{1.5cm} \includegraphics{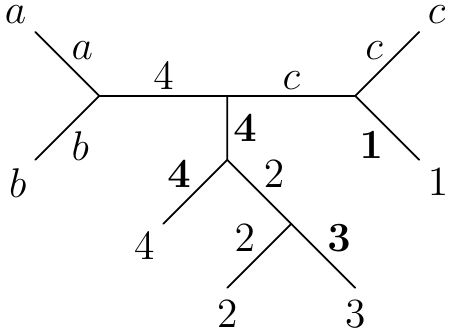}
    \caption{From left to right: A leaf-labeled trivalent tree $T$, its initial labeling of the leaf edges, and its full lazy tournament edge labeling.  Winners of each round of the tournament are shown in boldface at right, indicating $T \in \Tour(1, 0, 1, 2)$.}
    \label{fig:example-tournament}
\end{figure}

\begin{definition}\label{def:lazy-tournament}
Let $T$ be a leaf-labeled trivalent tree. The \textbf{lazy tournament} of $T$ is a labeling of the edges of $T$ computed as follows.  Start by labeling each leaf edge (that is, an edge adjacent to a leaf vertex) by the value on the corresponding leaf, as in the second picture of Figure \ref{fig:example-tournament}.  Then iterate the following process:
\begin{enumerate}
    \item \textbf{Identify which pair `face off'.} Among all pairs of labeled edges $(i,j)$ (ordered so that $i<j$) that share a vertex and have a third unlabeled edge $E$ attached to that vertex, choose the pair with the largest value of $i$.
    \item \textbf{Determine the winner.}  The larger number $j$ is the \textit{winner}, and the smaller number $i$ is the \textit{loser} of the match.
    \item \textbf{Determine which of $i$ or $j$ advances to the next round.}  Label $E$ by either $i$ or $j$ as follows:
    \begin{enumerate}
        \item If $E$ is adjacent to a labeled edge $u\neq j$ with $u>i$, then label $E$ by $i$.  (We say $i$ \textit{advances}.)
        \item Otherwise, label $E$ by $j$. (We say $j$ \textit{advances}.)
    \end{enumerate}  
\end{enumerate}
We then repeat steps 1-3 until all edges of the tree are labeled. 
\end{definition}

We refer to Step 3(a) above as the \textbf{laziness rule}, since $j$ drops out of the tournament despite winning its match. This happens when $j$ can see that its opponent $i$ will be defeated, again, in its next round against $u$. 

An example of the result of the lazy tournament process is shown in Figure \ref{fig:example-tournament}.  For a more detailed example, see Example \ref{ex:tournament} below.

\begin{remark}
The lazy tournament algorithm is geometrically motivated by the string equation, which leads to a recursion for the multidegrees discovered in \cite{CGM}. The string equation makes use of certain forgetting maps $$\pi_i:\Mbar_{0,\{a,b,c,1,2,\ldots,n\}}\to \Mbar_{0,\{a,b,c,1,2,\ldots,i-1,i+1,\ldots,n\}}$$ induced by forgetting the $i$-th marked point.  Similarly, when a label $j$ advances in a tournament after playing against label $i$, it can be thought of as applying the forgetting map $\pi_i$ to the corresponding stable curve and then relabeling the remaining points. We make this precise in Remark~\ref{rmk:forget-and-forgive}; see also Example \ref{ex:forget}. \end{remark}

We can now state the main result, which says that the multidegrees may be enumerated by keeping track of the winners in all possible tournaments.

\begin{definition}\label{def:tour}
For any weak composition $\mathbf{k} = (k_1,\dots, k_n)$ of $n$, let $\Tour(\mathbf{k})$ be the set of trivalent trees whose leaves are labeled by $\{a,b,c,1,\dots, n\}$, in which (a) the leaf edges $a$ and $b$ share a vertex, and (b) each label $i\geq 1$ wins exactly $k_i$ times in the tournament.
\end{definition}

For $n=0$, we write $\mathbf{k} = \emptyset$ for the empty composition and we have $\Tour(\emptyset) = \{T_0\}$, the unique such trivalent tree. In Figure \ref{fig:example-tournament}, the tree $T$ is in $\Tour(1, 0, 1, 2)$.

\begin{thm}\label{thm:tournaments}
 We have $\deg_{\mathbf{k}}(\emb_n)=|\Tour(\mathbf{k})|.$
\end{thm}

As $\mathbf{k}$ varies over all compositions, the sets $\Tour(\mathbf{k})$ partition the complete set of boundary points on the divisor $\delta_{a,b} \cong \Mbar_{0, \{b, c, 1, \ldots, n\}}$, consisting of curves in which $a, b$ are alone on the same component. The boundary points on this divisor correspond to trivalent trees on $n+2$ vertices, of which there are $(2n-1)!!$, which gives the following corollary of Theorem \ref{thm:tournaments}.

\begin{corollary}\label{cor:total-degree}
  The total degree of the embedding $\emb_n:\Mbar_{0,n+3}\to \PP^1\times \PP^2 \times \cdots \times \PP^n$ is $$\deg(\emb_n)=\sum_{k_1+\cdots+k_n=n}\deg_{(k_1,\ldots,k_n)}(\emb_n)=\sum_{k_1+\cdots+k_n=n}|\Tour(k_1,\ldots,k_n)|=(2n-1)!!$$ for all $n$.
\end{corollary}

Note that equation (\ref{eq:multidegree}) means that a generic choice of hyperplanes (with $k_i$ hyperplanes taken from the $\PP^i$ component for each $i$) intersects the image of the embedding $\emb_n$ in exactly $\deg_{(k_1,\ldots,k_n)}(\emb_n)$ points.   Since the set $\Tour(\mathbf{k})$ of trees represents a set of $\deg_{(k_1,\ldots,k_n)}(\emb_n)$ boundary points on $\Mbar_{0,n+3}$, it is natural to ask if there is a set of hyperplanes whose intersection with $\emb_n(\Mbar_{0,n+3})$ is $\Tour(\mathbf{k})$. In general there are not, as the example below shows.

\begin{example}\label{ex:intro}
The two trees in $\Tour(1,1)$, shown in Figure \ref{fig:tour11}, represent boundary points whose coordinates (with the conventions of Section \ref{sec:Kapranov}) under the map $\emb_2$ in $\PP^1\times \PP^2$ are ${[0:1]\times [0:1:0]}$ and ${[0:1]\times [0:1:1]}$. In any choice of hyperplanes $H_1$ from $\PP^1$ and $H_2$ from $\PP^2$ intersecting the embedding in these two points, we must have that $H_1$ is the single point $[0:1]$ in $\PP^1$, and $H_2$ must be the line $[0:1:t]$ in $\PP^2$.  However, the intersection $[0:1] \times [0:1:t]$ actually lies entirely on $\emb_2(\Mbar_{0,5})$, so the intersection is not transverse.  See Example \ref{ex:tour11} in Section \ref{sec:background-embedding} for more details.
\end{example}

\begin{figure}
    \centering
    \includegraphics{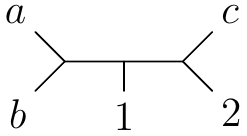} \hspace{2cm} \includegraphics{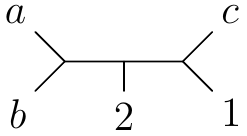}
    \caption{The two trees in $\Tour(1,1)$.  In both, the $c$ advances by the laziness rule on the first round, and is defeated once by each of $1$ and $2$.}
    \label{fig:tour11}
\end{figure}

Nonetheless, in future work we will show that the tournament points can sometimes be given as \textit{limits} of such intersections.  Towards this aim, we include in this paper the weaker result that there is always a set of hyperplanes, $k_i$ of them from $\PP^i$ for each $i$, whose intersection with $\emb_n(\Mbar_{0,n+3})$ \textit{contains} all the points in $\Tour(\mathbf{k})$.

\begin{thm}\label{thm:simple-hyperplanes}
Let $[z_b:z_c:z_1:z_2:\cdots:z_{r-1}]$ be the projective coordinates of the $\PP^r$ factor in $\PP^1\times \cdots \times \PP^n$ (with the conventions of Section \ref{sec:Kapranov}).  Then the coordinates of the points of $\Tour(k_1,\ldots,k_n)$ in the $\PP^r$ component all lie on the $k_r$ hyperplanes
$$z_b=0,\,\, z_c=0,\,\, z_1=0,\,\, \ldots,\,\, z_{k_r-2}=0,$$ where if $k_r=1$ then our collection only contains the hyperplane $z_b=0$, and if $k_r=2$ then we only have the two hyperplanes $z_b=0$ and $z_c=0$.  (If $k_r=0$ it is the empty collection.)
\end{thm}

The remainder of the paper is organized as follows.  In Section \ref{sec:background}, we provide some necessary background and definitions on the geometry of $\Mbar_{0,n}$.  In Section \ref{sec:proof}, we examine lazy tournaments and prove Theorem \ref{thm:tournaments}. In Section \ref{sec:PFs}, we give a direct bijection between the lazy tournaments and the column restricted parking functions defined in \cite{CGM}. Finally, in Section \ref{sec:hyperplanes}, we prove Theorem \ref{thm:simple-hyperplanes}.

\subsection{Acknowledgments}

We thank Renzo Cavalieri and Mark Shoemaker for helpful conversations pertaining to this work. 

\section{Background}\label{sec:background}

\subsection{Structure of $\Mbar_{0,X}$ and trivalent trees}

Let $X=\{a, b, c, 1, \ldots, n\}$. A point of $\Mbar_{0,X}$ consists of an (isomorphism class of a) genus-$0$ curve $C$ with at most nodal singularities and distinct, smooth marked points $p_i \in C$ labeled by the elements $i \in X$, such that each irreducible component has at least three \textbf{special points}, defined as marked points or nodes. The \textbf{dual tree} of a point of $\Mbar_{0,X}$ is the graph consisting of an unlabeled vertex for each irreducible component $C' \subseteq C$, a vertex labeled $i$ for each $i \in X$, and edges connecting $i$ and the vertex corresponding to $C'$ when $p_i \in C'$, and connecting $C'$ and $C''$ when $C'$ and $C''$ meet at a node. The resulting graph is always a tree since the curve has genus $0$. (See Figure \ref{fig:strata}).

A tree is \textbf{trivalent} if every vertex has degree $1$ or $3$ and at least one vertex has degree $3$. A tree is \textbf{at least trivalent} if it has no vertices of degree $2$ and at least one vertex of degree $\ge 3$.  Notice that the dual tree of any stable curve is at least trivalent.

Let $\Gamma$ be an at-least-trivalent tree whose leaves are labeled by $X$.  Then the \textbf{boundary stratum} $D_\Gamma$ corresponding to $\Gamma$ is the set of all stable curves whose dual tree is $\Gamma$.  The boundary strata $D_\Gamma$ form a quasi-affine stratification (as defined in \cite{3264}) of $\Mbar_{0,X}$, and the zero-dimensional boundary strata, or \textbf{boundary points}, correspond bijectively to the trivalent trees on leaf set $X$.  Indeed, since the points are isomorphism classes of stable curves and an automorphism of $\PP^1$ is determined by where it sends three points, a stable curve whose dual tree is trivalent represents the only element of its isomorphism class.

When $\Gamma$ has exactly two vertices of $v, w$ of degree $\geq 3$, the closure $\overline{D_{\Gamma}}$ is of codimension one, called a \textbf{boundary divisor}. For $i, j \in X$, we write $\delta_{i, j}$ for the boundary divisor with $i, j$ on $v$ and all other leaves adjacent to $w$.

\begin{figure}
    \centering
    \includegraphics{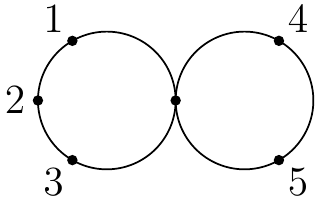}\hspace{1cm} \includegraphics{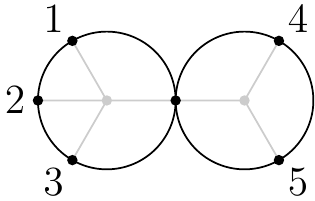}\hspace{1cm} \includegraphics{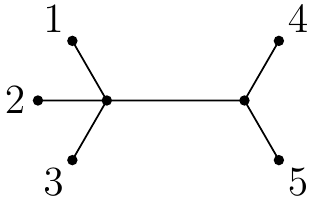}
    \caption{At left, a stable curve in $\Mbar_{0,5}$, in which each circle represents a copy of $\PP^1$.  At center, we form the dual tree $\Gamma$ of the curve, by drawing a vertex in the center of each circle and then connecting it to each marked point and adjacent circle. At right, we show $\Gamma$.  The set $D_\Gamma$ is the dimension-$1$ boundary stratum consisting of all stable curves in which $1,2,3$ are on one component and $4,5$ are on another.}
    \label{fig:strata}
\end{figure}

\subsection{The Kapranov morphism $\Mbar_{0,n+3} \to \mathbb{P}^n$}\label{sec:Kapranov}

For all facts stated throughout the next two subsections (\ref{sec:Kapranov} and \ref{sec:background-embedding}), we refer the reader to Kapranov's paper \cite{Ka1}, in which the Kapranov morphism below was originally defined.

The \textbf{$n$th cotangent line bundle} $\mathbb{L}_n$ on $\Mbar_{0,X}$ is the line bundle whose fiber over a curve $C\in \Mbar_{0,X}$ is the cotangent space of $C$ at the marked point $n$. The $n$-th \textit{$\psi$ class} is the first Chern class of this line bundle, written $\psi_n=c_1(\mathbb{L}_n)$. The corresponding map to projective space \[\kapn : \Mbar_{0,X} \to \mathbb{P}^n,\]
is called the Kapranov morphism.

We coordinatize this map as follows. It is known that $\kapn$ contracts each of the $n+2$ divisors $\delta_{n,i}$, for $i \ne n$, to a point $\bdpt_i := \kapn(\delta_{n,i}) \in \mathbb{P}^n$. These points are, moreover, in general linear position.  We choose coordinates so that $\bdpt_b, \bdpt_c, \bdpt_1, \ldots, \bdpt_{n-1} \in \mathbb{P}^n$ are the standard coordinate points $[1:\cdots : 0], \ldots, [0:\cdots:1]$ and $\bdpt_a$ is the barycenter $[1 : 1 : \cdots : 1]$. We name the projective coordinates $[z_b : z_c : z_1 : \cdots : z_{n-1}]$.
   
Given a curve $C$ in the interior $M_{0,X}$, by abuse of notation we also write $p_a,p_b,p_c,p_1,\dots,p_n$ for the coordinates of the $n+3$ marked points on the unique component of $C$, after choosing an isomorphism $C\cong \bP^1$. With these coordinates, the restriction of $\kapn$ to the interior $M_{0,X}$ is given by
\begin{equation} \label{eq:kap-interior}
    \kapn(C) = \bigg[
    \frac{p_a - p_b}{p_n - p_b} : 
    \frac{p_a - p_c}{p_n - p_c} : 
    \frac{p_a - p_1}{p_n - p_1} : 
    \cdots :
    \frac{p_a - p_{n-1}}{p_n - p_{n-1}}
    \bigg].
\end{equation}
It is often convenient to choose coordinates on $C$ in which $p_a = 0$ and $p_n = \infty$, in which case the map simplifies to
\[\kapn(C) = [p_b : p_c : p_1 : \cdots  : p_{n-1}].\]

With this coordinatization, we can take limits from the interior to obtain coordinates of $\kapn(C)$ for $C$ on the boundary of $\Mbar_{0,X}$.  In particular, consider the boundary stratum given by the dual graph:

\begin{center}
  \includegraphics{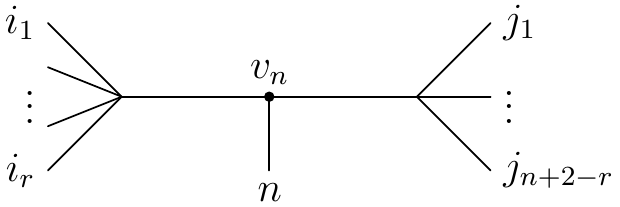}
\end{center} Here $\{a,b,c,1,2,\ldots,n-1\}=I\sqcup J$ where $I=\{i_1,\ldots,i_r\}$ and $J=\{j_1,\ldots,j_{n+2-r}\}$. Without loss of generality suppose $i_1=a$.  The points in this stratum may be obtained by taking a limit having the property that the points $p_{i}$ for $i\in I\setminus \{a\}$ approach $p_a=0$ and the points $p_j$ for $j\in J$ all approach $1$ (and $p_n=\infty$).  Thus, any point in this stratum has coordinates $[p_b:p_c:p_1:\cdots :p_{n-1}]$ where $p_i=0$ if $i\in I$ and $p_{j}=1$ if $j\in J$.  Now, since any trivalent tree is in the closure of a unique such stratum (by considering the two branches connected to the leaf edge $n$), we obtain the following.

\begin{lemma}\label{lem:coordinates}
  Let $C$ be a boundary point of $\Mbar_{0,X}$ corresponding to the trivalent tree $T$.  Let $v_n$ be the internal vertex of $T$ adjacent to the leaf edge whose leaf is labeled $n$, and consider the two remaining branches of $T$ connected to $v_n$.  Then $\kapn(C)=[z_b:z_c:z_1:z_2:\cdots:z_n]$ where $z_i=0$ if the leaf $i$ is on the same branch as the leaf $a$, and $z_i=1$ otherwise. 
\end{lemma}

\begin{example}
 Consider the tree below.
 \begin{center}
     \includegraphics{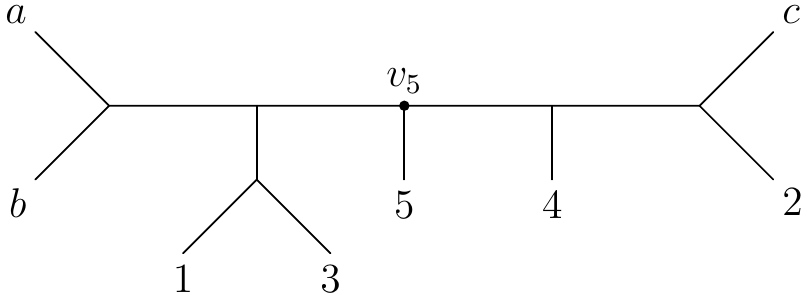}
 \end{center}
 The $5$ is the largest leaf, and it is connected to a vertex $v_5$ that is in turn connected to two other branches, one to the left of $v_5$, and one to the right of $v_5$.  The branch on the left contains $a$, so we have $z_b = z_1 = z_3 = 0$ and, from the other branch, $z_c = z_2 = z_4 = 1$.  Hence the tree maps under the Kapranov map to the point $$[z_b:z_c:z_1:z_2:z_3:z_4]=[0:1:0:1:0:1] \in \mathbb{P}^5.$$
\end{example}

\subsection{The iterated Kapranov embedding $\emb_n$}\label{sec:background-embedding}

Let $\pi_n : \Mbar_{0,X}\to \Mbar_{0,X\setminus n}$ be the $n$th forgetting map, which sends a stable curve $C$ to the stable curve $\pi_n(C)$ obtained by forgetting the point marked by $n$, and then collapsing any components with only two special points. If the dual tree of $C$ is $T$, then the dual tree of $\pi_n(C)$ is obtained from $T$ by deleting the label $n$ and its leaf, and then contracting any edges with degree $2$.

It is known that the Kapranov morphism, combined with $\pi_n$, gives a closed embedding
\begin{align*}
\Mbar_{0, X} &\hookrightarrow \mathbb{P}^n \times \Mbar_{0, X\setminus n}. \\
C &\mapsto \big( \kapn(C), \ \pi_n(C) \big).
\end{align*}

We may repeat this construction using the map $|\psi_{n-1}|$ on $\Mbar_{0, X\setminus n}$, and so on, obtaining a sequence of embeddings. This gives the {\bf iterated Kapranov morphism}
\[\emb_n: \Mbar_{0,X} \hookrightarrow \mathbb{P}^1 \times \mathbb{P}^2 \times \cdots \times \mathbb{P}^n.\]
The $i$-th factor of this embedding is given by forgetting the points $p_{i+1}, \ldots, p_n$, then applying the Kapranov morphism $|\psi_i|$ on the smaller moduli space. We can also combine the forgetting maps with Lemma \ref{lem:coordinates} to obtain the coordinates of any boundary point of $\Mbar_{0,X}$ under the embedding.

\begin{corollary}\label{cor:full-coordinates}
Let $C$ be a boundary point of $\Mbar_{0,X}$ corresponding to the trivalent tree $T$. Given an integer $1\leq r \leq n$, let $T'$ be the tree corresponding to $\pi_{r+1}\circ\pi_{r+2}\circ\cdots\circ\pi_n(C)$, let $v_r$ be the internal vertex of $T'$ adjacent to leaf edge $r$, and consider the three branches at $v_r$. Then the coordinates of $\emb_n(C)$ in the $\PP^r$ factor are $[z_b:z_c:z_1:\cdots:z_r]$ where $z_i=0$ if leaf $i$ is on the same branch as $a$ in $T'$, and $z_i=1$ otherwise.
\end{corollary}

\begin{example}\label{ex:tour11}
 Consider the two points in $\Tour(1,1)$, shown below.
 \begin{center}
     \includegraphics{tree-1-2.pdf}\hspace{2cm}\includegraphics{tree-2-1.pdf}
 \end{center}In the first, the $2$ separates $a,b,1$ from $c$, so it maps to $[0:1:0]$ in the $\PP^2$ factor.  In the second, the $2$ separates $a,b$ from $1,c$, so it maps to $[0:1:1]$ in the $\PP^2$ factor.
 
 In both, forgetting the point $2$ yields the tree:
 \begin{center}
 \includegraphics{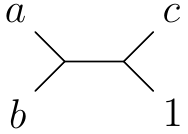}
 \end{center}
 Here, the $1$ separates $a,b$ from $c$, so both points map to $[0:1]$ in the $\PP^1$ factor.  We therefore obtain, as claimed in Example \ref{ex:intro}, that the two points in $\Tour(1,1)$ map to $[0:1]\times [0:1:0]$ and $[0:1]\times [0:1:1]$.
 
 Moreover, the only pair of hyperplanes from $\PP^1$ and $\PP^2$ that contains both defines the locus $[0:1]\times [0:1:t]$, which are the coordinates of the points on the stratum:
 \begin{center}
     \includegraphics{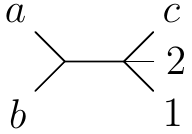}
 \end{center}
 since we may think of the point $1$ as varying on its component of the curve while $c$ and $2$ are fixed.
\end{example}

\subsection{Omega classes and a recursion for the multidegrees}

We now recall some facts about the multidegrees of the embedding $\emb_n$ proven in \cite{CGM}.  Throughout, let $\mathbf{k}=(k_1,\ldots,k_n)$ be a \textbf{weak composition} of $n$, that is, a sequence of nonnegative integers whose sum is $n$.

Recall from Equation \ref{eq:multidegree} that the $\mathbf{k}$-th \textbf{multidegree} of $\emb_n$ is the intersection number  $$\deg_{(k_1,\ldots,k_n)}(\emb_n)=\int_{\PP^1\times \cdots \times \PP^n}[\emb_n(\Mbar_{0,n+3})]\prod H_i^{k_i},$$ where $H_i$ is the class of a hyperplane in $\PP^i$ in the Chow ring of $\PP^1\times \cdots \times \PP^n$. Following \cite{CGM}, one can pull back the classes $H_i$ to the Chow ring of $\Mbar_{0,n+3}$ and express the multidegrees using \textbf{omega classes}, which are the pullbacks of hyperplane classes.

\begin{definition}[Omega classes]
Define $\omega_i=\emb_n^\ast H_i$ for all $i$.
\end{definition}

We can express multidegrees as intersections of omega classes: writing $\omega^\mathbf{k}$ for $\omega_1^{k_1} \cdots \omega_n^{k_n}$,
\[
\deg_{(k_1,\ldots,k_n)}(\emb_n)=\int_{\Mbar_{0,n+3}}\omega_1^{k_1}\omega_2^{k_2}\cdots \omega_n^{k_n} = \int_{\Mbar_{0,n+3}} \omega^\mathbf{k}.
\]
Furthermore, because the embedding $\emb_n$ is defined in terms of $\psi$ classes and forgetting maps \cite{CGM}, the omega classes are, equivalently, $\omega_i=f_i^\ast \psi_i$, where $f_i$ is the composite forgetting map \[f_i=\pi_{i+1}\circ \pi_{i+2} \circ \cdots \circ \pi_n:\Mbar_{0,n+3}\to \Mbar_{0,i+3}.\]
The multidegrees satisfy a recursion, stated as Proposition~\ref{prop:recursion} below.

\begin{definition}[Tilde Construction] \label{def:ktilde-j}
For any weak composition $\mathbf{k}=(k_1,\ldots,k_n)$ of $n$, let $i=\max\{\ell: k_\ell=0\}$ be the rightmost index such that $k_i=0$, where we set $i=c$ if $\mathbf{k}=(1,1,1,\ldots,1)$ is the unique composition of length $n$ with no $0$ entries. Let $j>i$, in the ordering $a<b<c<1<2<\cdots<n$.  

Then define $\widetilde{\mathbf{k}}_j$ to be the weak composition of size and length $n-1$ formed by (a) decreasing $k_j$ by $1$ and then (b) removing the rightmost $0$ in the resulting sequence (which is either the $i$-th or $j$-th entry).
\end{definition}

\begin{example}
 If $\mathbf{k}=(0,1,0,0,2,1,3)$, then $\widetilde{\mathbf{k}}_5=(0,1,0,1,1,3)$ since it is formed by subtracting one from the fifth entry, $2$, and removing the rightmost $0$.  On the other hand, $\widetilde{\mathbf{k}}_6=(0,1,0,0,2,3)$ since it is formed by subtracting one from the sixth entry, $1$, and then removing the new $0$.
\end{example}

\begin{prop}[Asymmetric string equation {\cite[Prop 4.10]{CGM}}]\label{prop:recursion}  Let $i$ be the index of the rightmost $0$ in $\mathbf{k}$.  Then the  multidegrees satisfy the recursion
\begin{equation} \label{eq:string}
    \deg_{\mathbf{k}}(\emb_n)=\sum_{j>i}\deg_{\widetilde{\mathbf{k}}_j}(\emb_{n-1}).
\end{equation}
In more detail, letting $\pi_i : \Mbar_{0, abc1\cdots n} \to \Mbar_{0, abc1\cdots\hat{i}\cdots n}$ be the forgetting map, 
\begin{align} \label{eq:stringeq1}
    \int_{\Mbar_{0, abc1\cdots n}} \omega^\mathbf{k} &= \int_{\Mbar_{0, abc1\cdots\hat{i}\cdots n}} (\pi_i)_*(\omega^\mathbf{k}) \\ \label{eq:stringeq2}
    &= \int_{\Mbar_{0, abc1\cdots\hat{i}\cdots n}} \sum_{j > i} \alpha_j \text{ for certain } \alpha_j \in H^*(\Mbar_{0, abc1\cdots\hat{i}\cdots n}), \\ \label{eq:stringeq3}
    &= \int_{\Mbar_{0, abc1\cdots n{-}1}} \sum_{j > i} \omega^{\widetilde{\mathbf{k}}_j},
\end{align}
where expression \eqref{eq:stringeq3} is obtained from \eqref{eq:stringeq2} by applying to each summand a \emph{different} identification $\Mbar_{0, abc1\cdots\hat{i}\cdots n} \cong \Mbar_{0, abc1\cdots n{-}1}$.
\end{prop}

\begin{example}[See {\cite[Ex. 4.6]{CGM}}] \label{ex:stringex}
For $\mathbf{k} = (1, 0, 0, 0, 2, 1, 3)$, the asymmetric string equation gives
\begin{align}
\int_{\Mbar_{0,abc1234567}}\omega_1\omega_5^2\omega_6\omega_7^3
 &= \int_{\Mbar_{0, abc123567}} (\pi_4)_*(\omega_1\omega_5^2\omega_6\omega_7^3) \\
\label{eq:stringex1}
 &= \int_{\Mbar_{0, abc123567}}\big(\omega_1(\psi_5\psi_6\psi_7^3) + \omega_1(\psi_5^2\psi_7^3) + \omega_1(\psi_5^2\psi_6\psi_7^2) \big)\\ \label{eq:stringex2}
&= \int_{\Mbar_{0,abc123456}}(\omega_1\omega_4\omega_5\omega_6^3 +
\omega_1\omega_5^2\omega_6^3 +
\omega_1\omega_4^2\omega_5\omega_6^2).
\end{align}
Line \eqref{eq:stringex1} uses, in part, the ordinary string equation for psi classes regarding $\pi_4$. Then, line \eqref{eq:stringex2} uses a different identification $\Mbar_{0, abc123567} \cong \Mbar_{0, abc123456}$ for each of the three terms. Later, Proposition \ref{prop:bijection} will imply that (in this example) the same operation $\pi_i$ and relabelings lead to a corresponding bijection
\begin{equation*}
\Tour(1,0,0,0,2,1,3) \xrightarrow{\ \sim\ } \Tour(1,0,0,1,1,3) \sqcup \Tour(1,0,0,0,2,3) \sqcup \Tour(1,0,0,2,1,2).
\end{equation*}
See Example~\ref{ex:forget} for full details.
\end{example}

The relabelings in \eqref{eq:stringeq2}--\eqref{eq:stringeq3} are stated precisely in Remark \ref{rmk:forget-and-forgive}, where we apply the same relabeling maps (along with $\pi_i$) to tournament points, as part of our proof of the multidegree formula in Section \ref{sec:proof}. Although we will continue to work with trees and point sets rather than cohomology classes, these forgetting maps and relabelings directly inspired the definition of lazy tournament.

Note that the recursion \eqref{eq:string} in Proposition \ref{prop:recursion} is similar to the recursion for the multinomial coefficients $\binom{n}{k_1,\ldots,k_m}=\frac{n!}{k_1!k_2!\cdots k_m!}$.  Recall that for any composition $(k_1,\ldots,k_m)$ of $n$ with all parts nonzero, we have $$\binom{n}{k_1,\ldots,k_m}=\sum_{j=1}^m \binom{n}{k_1,k_2,\ldots,k_{j-1},k_j-1,k_{j+1},\ldots,k_m}.$$  Defining the \textbf{asymmetric multinomial coefficient} to be the corresponding multidegree, that is, $$\multichoose{n}{k_1,\ldots,k_n}=\deg_{(k_1,\ldots,k_n)}(\emb_n),$$ we can restate Proposition \ref{prop:recursion} as $$\multichoose{n}{\bf k}=\sum_{j>i} \multichoose{n}{\widetilde{\bf k}_j}$$
for any weak composition $\mathbf{k}$ of $n$ into $n$ parts.

\section{Tournaments and the proof of Theorem \ref{thm:tournaments}}\label{sec:proof}

The main goal of this section is to prove Theorem \ref{thm:tournaments}, which we restate for convenience below. We say an edge of a tree is a \defn{leaf edge} if it is adjacent to a leaf vertex.  Also recall from Definition \ref{def:tour} that $\Tour(k_1,\ldots,k_n)$ is the set of trivalent trees $T$ whose leaves are labeled by $\{a,b,c,1,\dots, n\}$ in which the leaf edges $a$ and $b$ share a vertex, and each label $i\geq 1$ wins exactly $k_i$ times in the lazy tournament of $T$.

\begin{TournamentsThm}
   For any weak composition $\mathbf{k}=(k_1,\ldots,k_n)$ of $n$, we have $$\deg_{\mathbf{k}}(\emb_n)=|\Tour(\mathbf{k})|.$$
\end{TournamentsThm}
   
Recall the definition of a lazy tournament from Definition \ref{def:lazy-tournament}. We begin this section with an illustrative example of the tournament process.

\begin{example}\label{ex:tournament}
Start with the following tree, in which we have labeled each leaf edge by the label of the corresponding leaf:
\begin{center}
\includegraphics{small-tree-edge-labels.pdf}
\end{center}
To start the tournament, we first compare all three pairs of leaves whose leaf edges share a vertex: $(a,b)$, $(2,3)$, and $(c,1)$.  The one with the largest smaller entry is $(2,3)$ (since the chosen ordering is $a<b<c<1<2<3<4$) and so they face off first.  The number $3$ \emph{wins}, and by the laziness rule with $u=4$, the number $2$ \emph{advances}.  We highlight the winner in boldface and draw the new label:
\begin{center}
\includegraphics{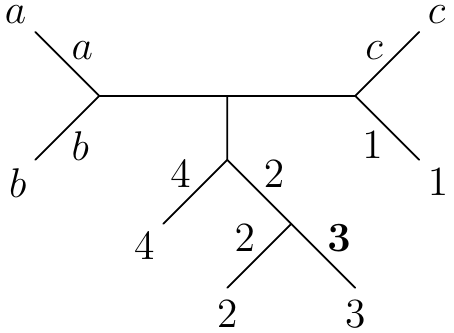}
\end{center}
For the next round, our possible pairs are $(a,b)$, $(c,1)$, and $(2,4)$, so $2$ and $4$ face off next.  The entry $4$ wins (shown in boldface) and advances since the laziness rule does not apply:
\begin{center}
\includegraphics{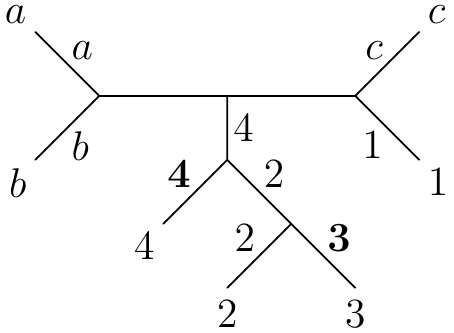}
\end{center}
For the next round, our possible pairs are just $(a,b)$ and $(c,1)$, and the one with the larger smaller entry is $(c,1)$.  So $1$ and $c$ face off, with $1$ winning and $c$ advancing by the laziness rule with $u=4$:
\begin{center}
\includegraphics{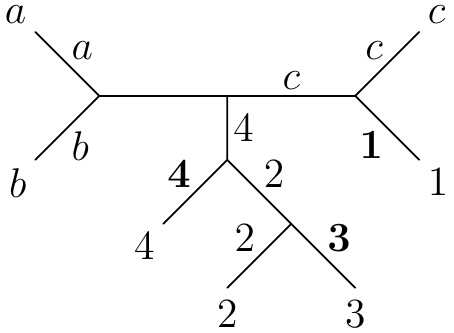}
\end{center}
In the final round, $c$ faces off against $4$, and $4$ wins and advances:
\begin{center}
\includegraphics{small-tree-tournament.pdf}
\end{center}
Since $1$ won one round, $3$ won one round, and $4$ won two rounds, this tree is an element of $\Tour(1,0,1,2)$.
\end{example}

\begin{remark}
A trivalent tree with $n+3$ labeled leaves has exactly $n$ rounds in its lazy tournament.
\end{remark}

\begin{remark}
The set of \emph{all} trivalent trees with leaves labeled $a, b, c, 1, \ldots, n$, such that the leaf edges of $a$ and $b$ share a vertex, is the disjoint union of $\Tour(\mathbf{k})$ over all weak compositions $\mathbf{k} = (k_1, \ldots, k_n)$ of $n$.
\end{remark}

Before embarking on the proof of Theorem \ref{thm:tournaments}, we illustrate it in the case $n=2$.

\begin{example}
 The three leaf-labeled trivalent trees on $\{a,b,c,1,2\}$ in which the leaf edges of $a$ and $b$ share a vertex are shown below:
 \begin{center}
 \includegraphics{tree-1-2.pdf}\hspace{2cm} \includegraphics{tree-2-1.pdf}\hspace{2cm} \includegraphics{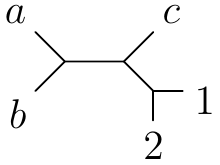}  
 \end{center}
 In the tournament of the first two, the labels $1$ and $2$ each win one round, and in the tournament of the third, the label $2$ wins both rounds.  Thus $\Tour(1,1)=2$ and $\Tour(0,2)=1$, and so by Theorem \ref{thm:tournaments} we have $\deg_{(1,1)}(\emb_2)=2$, $\deg_{(0,2)}(\emb_2)=1$, and $\deg_{(2,0)}(\emb_2)=0$.
\end{example}

\subsection{Combinatorial results about tournaments}

We now prove a number of technical lemmas about tournaments that will be helpful in the proof of Theorem~\ref{thm:tournaments}.  The first is that the losing elements of each round weakly decrease as the tournament is run.

\begin{lemma}\label{lem:losers-decrease}
  Let $T$ be a trivalent tree on $n+3$ leaves, and let $i_1,i_2,\ldots,i_n$ be the smaller elements of the pairs that face off in its lazy tournament in each round, listed in order from start to finish. Then we have $i_1\ge i_2 \ge \cdots \ge i_n$.
\end{lemma}

\begin{proof}
  We show that $i_1\ge i_2$; the remaining inequalities follow by induction on $n$.
  
  The first pair of leaves to face off in the tournament is $(i_1,j)$ for some $j>i_1$, and by the definition of the tournament, all other pairs of leaves $(i',j')$ have $i'<i_1$.  Let $E$ be the third edge adjacent to leaf edges $i_1$ and $j$, as defined in Definition~\ref{def:lazy-tournament}.
  
  \textbf{Case 1:} Suppose $E$ is not adjacent to any other labeled edge besides $i_1$ and $j$.  Then after the first round, the new possible pairs to consider for determining who faces off next are all of the form $(i',j')$ with $i'<i_1$.  Since $i_2$ is among the values $i'$, we have $i_2<i_1$.
  
  \textbf{Case 2:} Suppose $E$ is adjacent to two other labeled edges besides $i_1$ and $j$. Then we are at the last round of the tournament, and the statement of the lemma holds trivially.
  
  \textbf{Case 3:} Suppose $E$ is adjacent to exactly one other labeled edge $u$.  If $u<i_1$, then $j$ advances in the tournament and $(u,j)$ becomes one of the new pairs to consider along with the other $(i',j')$ pairs.  So in the next round, $i_2$ is either $u$ or one of the $i'$ values, so since $u<i_1$ and each $i'<i_1$, we again have $i_2<i_1$.

  Otherwise, if $u>i_1$, then by the laziness rule $i_1$ advances, and it becomes a new pair $(i_1,u)$ to consider for the next round; since $i_1$ is still largest among the smaller elements of each leaf pair $(i',j')$, we have $i_2=i_1$ in this case.
  
  Thus in all cases $i_1\ge i_2$, and similarly $i_1\ge i_2 \ge \cdots \ge i_n$.
\end{proof}

\begin{definition}
Given a label $i$ of a leaf-labeled trivalent tree in $\Tour(k_1,\dots,k_n)$, we say $i$ is a \textbf{winner} if it wins any round of the tournament, and it is a \textbf{loser} if it loses any round of the tournament. 
\end{definition}

The next two lemmas show that every label that participates in at least one match is either a winner or a loser but cannot be both. That is, ``winners always win'' and ``losers always lose''.

\begin{lemma}[Winners Lemma]\label{lem:winners-win}
Suppose $j$ is a label of a leaf-labeled trivalent  tree. If $j$ wins the first round in which it competes during the lazy tournament, then $j$ wins all subsequent rounds in which it competes.
\end{lemma}

\begin{proof}
  Suppose $j$ wins a round against $i<j$ and advances. Since the sequence of losers in the tournament decreases by Lemma~\ref{lem:losers-decrease}, then $j$ cannot appear after $i$ in the list of losers. Therefore, $j$ must win all of the subsequent rounds in which it competes.
\end{proof}

\begin{lemma}[Losers Lemma]\label{lem:losers-lose}
Suppose $i$ is a label of a leaf-labeled trivalent tree. If $i$ loses the first round in which it competes during the lazy tournament, then $i$ loses all subsequent rounds in which it competes.
\end{lemma}

\begin{proof}
  If $i$ loses but also advances to a future round, it must have done so via the laziness rule, and so it is already adjacent to another edge that is larger than it.  Thus it also loses its next round, and so on.
\end{proof}

We now restrict our attention to trivalent trees having $a,b$ adjacent.  Recall that $\Tour(\mathbf{k})$ is the set of trivalent trees whose leaves are labeled by $\{a,b,c,1,\dots, n\}$ in which the leaf edges $a$ and $b$ share a vertex, and each label $i\geq 1$ wins exactly $k_i$ times in the tournament.

\begin{lemma}[Participation Lemma]\label{lem:numbers-used}
 Let $T\in \Tour(\mathbf{k})$ for a weak composition $\mathbf{k}$ of $n\geq 1$.  In the tournament of $T$, neither $a$ nor $b$ compete in any round, and every other label competes in some round.  Moreover, when a label advances, it always advances forward along its path towards $a$.
 \end{lemma}

\begin{proof}
  Since we use the ordering $a<b<c<1<2<3<\cdots<n$ and $a,b$ are adjacent leaves in $T$, we see that $a$ will never be the largest of the smaller entries of the available pairs at any step of the tournament, and so $a$ and $b$ never face off against each other before the tournament ends.
   
  Now, orient each edge of $T$ towards $a$ (by considering its unique path to $a$).  This makes the set of unlabeled edges, which starts as the set of non-leaf edges, into an oriented rooted tree, rooted at the vertex $v_a$ adjacent to $a,b$.  After each round, some leaf of this subtree becomes labeled, and the remaining unlabeled edges again form a connected oriented tree rooted at $v_a$.  Inductively, we see that the labels advance forward along their paths towards $a$.
  
  Finally, let $\ell\neq a,b$ be a labeled leaf in $T$, let $e$ be the next edge after leaf edge $\ell$ along its path to $a$, and let $f$ be the third edge adjacent to edges $\ell$ and $e$.  Then eventually $f$ is labeled by some label $j$ by the process above, and then $\ell$ faces off against $j$ to label edge $e$.  Thus every label $\ell\neq a,b$ competes in some round.
\end{proof}

Our final result below is that the largest index $i$ for which $k_i=0$, which is used in the definition of $\widetilde{\mathbf{k}}_j$ (Definition \ref{def:ktilde-j}) and the asymmetric string equation, is the largest (and first) loser in the tournament.   Moreover, we are in the lazy case (or not) depending on whether $k_j=1$ or $k_j>1$.

\begin{lemma}[First Round Lemma]\label{lem:biggest-loser}
   Let $T\in \Tour(\mathbf{k})$ and consider the pair $(i_1,j)$ in the tournament of $T$ that faces off first, written so that $i_1<j$.  
\begin{enumerate}
    \item Let $i$ be the largest index for which $k_i=0$.  Then $i_1=i$.  
   \item The laziness rule applies in this round if and only if $k_j=1$ and $n\ge 2$.
\end{enumerate} 
\end{lemma}

\begin{proof}
  1. Since $i_1$ loses every round it competes in (by Lemma \ref{lem:losers-lose}), we have $k_{i_1}=0$ by the definition of $\Tour(\mathbf{k})$.  Furthermore, by Lemma \ref{lem:winners-win} and \ref{lem:numbers-used}, the only indices $i'$ with $k_{i'}=0$ are those that lose a round at some point.  By Lemma \ref{lem:losers-decrease}, these indices decrease as the tournament is run, so $i_1$ is the largest index for which $k_{i_1}=0$, and so $i_1=i$. 
  
  2. If the laziness rule applies, then $j$ does not advance, so evidently $k_j = 1$. Conversely, suppose $k_j = 1$ and (for contradiction) that $j$ advances. By the Winners Lemma (\ref{lem:winners-win}), $j$ wins every round in which it competes, and by the Participation Lemma (\ref{lem:numbers-used}), $j$ will compete in one more round unless $n=1$. This gives a contradiction unless $n=1$.
\end{proof}

\subsection{Proof of Theorem \ref{thm:tournaments}}

In order to prove Theorem \ref{thm:tournaments}, we will show that the quantities $|\Tour(k_1,\ldots,k_n)|$ satisfy the same recursion as the multidegrees (Proposition \ref{prop:recursion}).  That is, we wish to show that for all $n\ge 1$ and any composition $\mathbf{k}=(k_1,\ldots,k_n)$ of $n$ where $k_i$ is the rightmost $0$ in $\mathbf{k}$, $$|\Tour(\mathbf{k})|=\sum_{j>i} |\Tour(\widetilde{\mathbf{k}}_j)|.$$
To do so, we use the tournament process to define a map $$\forget:\Tour(\mathbf{k})\to \coprod_{j>i} \Tour(\widetilde{\mathbf{k}}_j),$$ and we show $\forget$ is a bijection. Here, $\coprod$ is coproduct (formal disjoint union) of sets, since the sets $\Tour(\widetilde{\mathbf{k}}_j)$ may not be pairwise disjoint.

\begin{definition}
  Define $\forget:\Tour(\mathbf{k})\to \coprod_{j>i} \Tour(\widetilde{\mathbf{k}}_j)$ as follows.  Let $T\in \Tour(\mathbf{k})$ and consider the pair $(i,j)$ in the tournament of $T$ that faces off first, with $j>i$. Let $v$ be the vertex adjacent to $i$ and $j$.
  \begin{itemize}
      \item If $k_j>1$, then define $\forget(T)$ to be the tree formed by labeling $v$ by $j$, removing the leaves and leaf edges of $i$ and $j$, and decreasing all of the labels $i+1,i+2,\ldots,n$ by $1$. 

      \item If $k_j=1$, then define $\forget(T)$ to be the tree formed by labeling $v$ by $i$, removing the leaves and leaf edges of $i$ and $j$, and decreasing all of the labels $j+1,j+2,\ldots,n$ by $1$. 
  \end{itemize}
By Lemma~\ref{lem:biggest-loser}(2), the two cases above correspond to whether or not the laziness rule applies to the first round of the tournament. Below, in Lemma~\ref{lem:map}, we verify that this gives $\forget(T) \in \Tour(\widetilde{\mathbf{k}}_j)$, i.e. the composition associated to $\forget(T)$ is $\widetilde{\mathbf{k}}_j$. An example of $\forget$ is shown in Figure~\ref{fig:F}.
\end{definition}

\begin{remark}[Forgetting maps]\label{rmk:forget-and-forgive}
  Geometrically, $\forget(T)$ is given by applying the forgetting map $\pi_i : \Mbar_{0, X} \to \Mbar_{0, X\setminus i},$ followed by one of two different relabelings of the marked points, as follows. In the non-lazy case (when $j$ advances), after applying $\pi_i$ we simply decrease each of the labels $i+1,\ldots,n$ by one. In the lazy case (when $i$ advances), after applying $\pi_i$ we relabel $j$ as $i$ and then decrease each of the labels $j+1,\ldots,n$ by one. The same relabelings were used in the geometric proof of the asymmetric string equation in \cite{CGM} (restated here as Proposition \ref{prop:recursion}).
  \end{remark}
  
\begin{example} \label{ex:forget} Continuing Example \ref{ex:stringex}, the following recursion holds:
\begin{equation}
\Tour(1,0,0,0,2,1,3) \xrightarrow{\forget} \Tour(1,0,0,1,1,3)   \amalg \Tour(1,0,0,0,2,3) \amalg \Tour(1,0,0,2,1,2).
\end{equation}
Geometrically, $\forget$ maps the corresponding points of $\Mbar_{0,n+3}$ to $\Mbar_{0,n+2}$ as in the diagram
\begin{equation}
\raisebox{4em}{\xymatrix{
&& \Mbar_{0,abc123456} \ar@{}[r]|-*[@]{\supset} & \Tour(1,0,0,1,1,3) \\
\Mbar_{0, abc1234567} \ar[r]^{\pi_4} & \Mbar_{0,abc123567}
\ar@{=}[ru]^{\cong}
\ar@{=}[r]^{(\cong)'}
\ar@{=}[rd]^{\cong} &
\Mbar_{0,abc123456}
\ar@{}[r]|-*[@]{\supset} & \Tour(1,0,0,0,2,3) \\
\Tour(1,0,0,0,2,1,3)
\ar@{}[u]|-*[@]{\cup}
&& \Mbar_{0,abc123456}
\ar@{}[r]|-*[@]{\supset} & \Tour(1,0,0,2,1,2).
}}
\end{equation}
Here, the relabelings indicated by $\cong$ are \emph{non-lazy} (decrementing $5$, $6$ and $7$ by $1$) and the middle relabeling $(\cong)'$ is \emph{lazy}: it is given by first sending $6 \mapsto 4$, then decrementing only $7 \mapsto 6$. The map $\forget$ corresponds to these maps, applied to the appropriate subsets of $\Tour(1, 0, 0, 0, 2, 1, 3)$. Note that these are the same maps as in the asymmetric string equation.
\end{example}

\begin{figure}
    \centering
    \includegraphics{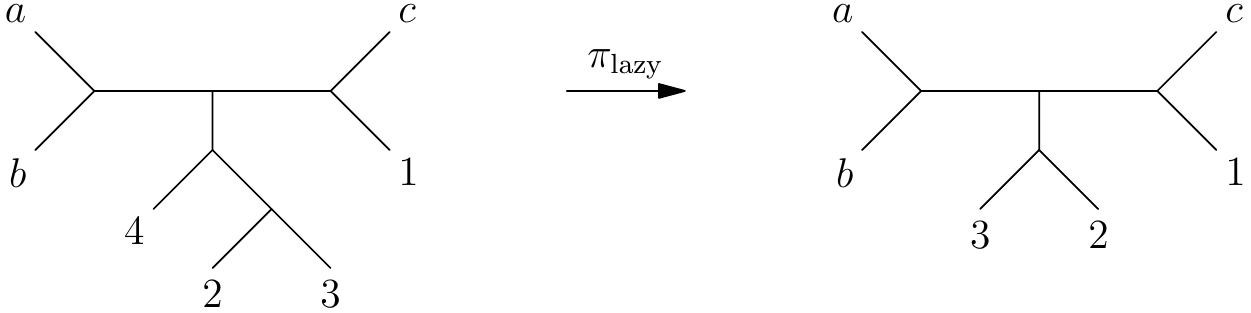}
    \caption{The map $\forget$ applied to the tree from Example \ref{ex:tournament}.  Notice that applying $\forget$ corresponds to running the first round of the tournament, deleting the two leaf edges that competed, and decrementing the higher label $4$.}
    \label{fig:F}
\end{figure}

\begin{lemma}\label{lem:map}
  The map $\forget$ is a well-defined map to $\coprod_{j>i} \Tour(\widetilde{\mathbf{k}}_j)$; in particular, if $T\in \Tour(\mathbf{k})$ and $j$ is the winner of the first round of the tournament of $T$, then $\forget(T)\in \Tour(\widetilde{\mathbf{k}}_j)$.
\end{lemma}

\begin{proof}
  If $n=1$, the only composition is $\mathbf{k}=(1)$, and $\Tour(1)$ consists of the unique trivalent tree $T_1$ with leaf edges $a,b$ sharing one vertex and $1,c$ sharing another.  Then $\forget(T_1)$ is the unique tree $T_0\in \Tour(\emptyset)$ in which $a,b,c$ are leaf edges connected to a single internal vertex. 
 
  Now suppose $n>1$.  Let $T\in \Tour(\mathbf{k})$, and let $j$ be the winner of the first round, which faces off against $i<j$. By Lemma~\ref{lem:biggest-loser}(1), $i$ is the largest index such that $k_i = 0$.  Furthermore, by Lemma~\ref{lem:biggest-loser}(2), the laziness rule applies and $i$ advances if and only if $k_j=1$ (since $n>1$).
  
  Let $\mathbf{k}'$ be the composition associated to $\forget(T)$'s tournament, that is, $k'_\ell$ is the number of rounds of the tournament of $\forget(T)$ that $\ell$ wins. We wish to show $\forget(T) \in \Tour(\widetilde{\mathbf{k}}_j)$, that is, $\mathbf{k}' = \widetilde{\mathbf{k}}_j$.
  
  \textbf{Case 1.} Suppose $k_j>1$. Then $\forget(T)$ is formed by removing the edge labeled $i$ from $T$ and decreasing the labels $i+1,i+2,\ldots,n$ by $1$.
  Since $\forget(T)$ can be thought of as the tree that remains after letting $j$ advance in the tournament (and renumbering), we have $k'_{j-1}=k_j-1$, $k_\ell'=k_\ell$ for $\ell<i$, and $k_\ell'=k_{\ell+1}$ for all other $\ell$ ($\ell\ge i$, $\ell\neq j-1$).  Thus $\mathbf{k}'=\widetilde{\mathbf{k}}_j$ as desired.
  
  \textbf{Case 2.} Suppose $k_j=1$.  Then $\forget(T)$ is formed by removing the edge labeled $j$ from $T$ and decreasing the labels $j+1,j+2,j+3,\ldots,n$ by $1$.  
  Since $\forget(T)$ can be thought of as the tree that remains after letting $i$ advance in the tournament via the laziness rule (and renumbering), we see that for all $\ell<j$ we have $k_\ell'=k_\ell$, and for all $\ell\ge j$ we have $k_\ell'=k_{\ell+1}$. Thus $\mathbf{k}'=\widetilde{\mathbf{k}}_j$.
\end{proof}

\begin{prop}\label{prop:bijection}
  The map $\forget$ is a bijection.
\end{prop}

\begin{proof}
  We construct the inverse of $\forget$.  Given an element $T'\in \Tour(\widetilde{\mathbf{k}}_j)$ for some $j>i$, we construct the unique $T\in \Tour(\mathbf{k})$ such that $\forget(T)=T'$ as follows.  
 
 \textbf{Case 1.} If $k_j> 1$, define $T$ by increasing the labels $i,i+1,i+2,\ldots,n$ by $1$ each in $T'$, splitting the leaf edge of $i$ into two edges with middle vertex $v$, and then attaching a leaf edge labeled $i$ to $v$.  By the definition of $\widetilde{\mathbf{k}}_j$, the rightmost $0$ in $\widetilde{\mathbf{k}}_j$ occurs strictly before $i$, so all losers in $T'$ are less than $i$.  Thus the pair $(i,j)$ is the first to face off in $T$.  Moreover, if there is a labeled edge $u$ adjacent to the empty edge connected to $(i,j)$ in $T$, assume for contradiction that $u>i$.  Then $(u-1,j-1)$ was a pair of leaves that faced off in $T'$, so $u-1<i$ and so $u\le i$, a contradiction.  Hence in $T$, $j$ advances after defeating $i$, and therefore $T\to T'$ under our map.
 
 \textbf{Case 2.} If $k_j=1$, define $T$ by first increasing the labels $j,j+1,\ldots,n$  by $1$ each in $T'$, splitting the leaf edge of $j$ into two edges with middle vertex $v$, and then attaching a leaf edge labeled $j$ to $v$.  Note that since $k_j=1$ we have $(\widetilde{\mathbf{k}}_j)_i=0$ by the definition of $\widetilde{\mathbf{k}}_j$, and moreover $i$ is the index of the last $0$ in $\widetilde{\mathbf{k}}_j$.  Thus $i$ was the loser of the first match in $T'$, meaning it was paired with a larger entry $u>i$ in the first round of $T'$.
 Thus in $T$, the laziness rule applies in the match between $i$ and $j$, and $j$ sends $i$ along to face off against $u$.   It follows that our map above sends $T\to T'$ as desired.
\end{proof}

\begin{proof}[Proof of Theorem~\ref{thm:tournaments}]
By Proposition~\ref{prop:bijection}, the quantities $\deg_{\mathbf{k}}(\Phi_n)$ and $|\Tour(\mathbf{k})|$ satisfy the same recursion. It is easy to check that $\deg_{1}(\Phi_1) = |\Tour(1)| = 1$. Therefore, by induction on $n$ we have that $\deg_{\mathbf{k}}(\Phi_n) = |\Tour(\mathbf{k})|$.
\end{proof}

Finally, we can use Theorem \ref{thm:tournaments} to give an elegant proof of Corollary \ref{cor:total-degree}, which we restate for convenience below.

\begin{TotalCor}
  The total degree of $\emb_n$ is $$\deg(\emb_n)=(2n-1)!!=(2n-1)\cdot (2n-3)\cdot \cdots \cdot 5 \cdot 3 \cdot 1.$$
\end{TotalCor}
 
\begin{proof}
  Notice that $\bigcup_{\mathbf{k}}\Tour(\mathbf{k})$ is the set $\Tour(n)$ of all trivalent trees with leaves labeled by $a,b,c,1,2,\ldots,n$ in which the leaf edges of $a$ and $b$ share a vertex. Since the sets $\Tour(\mathbf{k})$ are all disjoint as $\mathbf{k}$ varies, we have that the total degree is $|\Tour(n)|$, so we wish to show $|\Tour(n)|=(2n-1)!!$.
  
  Notice that the forgetting map $\pi_b$ maps $\Tour(n)$ bijectively onto the set of all boundary points on $\Mbar_{0,\{a,c,1,2,\ldots,n\}}\cong \Mbar_{0,n+2}$, i.e. all trivalent trees on $n+2$ labels.  It is well known that there are $(2(n+2)-5)!!=(2n-1)!!$ such trees, and the result follows.
\end{proof}

\begin{remark}
One can also prove that $|\Tour(n)|=(2n-1)!!$ via a direct inductive proof as follows. Since $|\Tour(0)|=1$, it suffices to give a $(2n-1)$-to-one map from $\Tour(n)$ to $\Tour(n-1)$.  This is precisely given applying the forgetting map $\pi_n$ to each of these points. Indeed, forgetting the $n$ results in an element of $\Tour(n-1)$, and given a tree in $\Tour(n-1)$, there are $2n-1$ ways to attach a new leaf edge labeled $n$ to form a tree in $\Tour(n)$ in its preimage: namely, either attached to the leaf labeled by one of the $n$ marked points $c,1,2,\ldots,n-1$, or attached to one of the $n-1$ non-leaf edges in the tree.  
\end{remark}

\section{Parking functions} \label{sec:PFs}

In \cite{CGM}, the multidegrees $\deg_{\bf k}(\emb_n)$ were shown to enumerate a type of parking function called a \textit{column restricted parking function} ($\CPF$). In this section we give a bijection between $\CPF$s and lazy tournaments.

We first recall the general definition of a parking function, as in \cite{HaglundBook}, though we draw our Dyck paths using up and left steps rather than up and right, as seen in, for instance, \cite{Blasiaketc}.  This will be useful for avoiding the extra step of reversing the sequence $\mathbf{k}$ (as in \cite{CGM}).

\begin{definition}
A \textbf{Dyck path} of height $n$ is a lattice path in the plane from $(0,n)$ to $(n,0)$, using right $(1,0)$ and down $(0,-1)$ unit steps, that stays weakly above the diagonal line connecting the two endpoints. 
\end{definition}

\begin{definition}
A \textbf{parking function} is a way of labeling the unit squares just to the left of the downward steps of a Dyck path with the numbers $1,2,\ldots,n$ such that the numbers in each column are increasing up the column. For $1\leq i\leq n$, the \textbf{$i$th column} of a parking function is the $i$th column of squares from the left in the $n\times n$ box that contains it. We write $\PF(n)$ for the set of parking functions of size $n$.
\end{definition}

Two examples of parking functions of height $4$ are shown in Figure \ref{fig:PF}.

\begin{figure}
    \centering
    \includegraphics{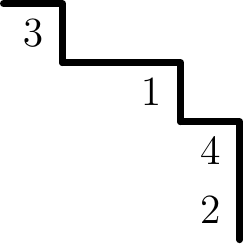}\hspace{2cm}\includegraphics{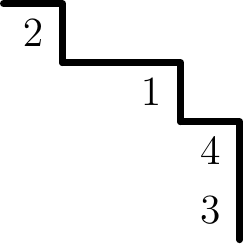}\hspace{2cm}\includegraphics{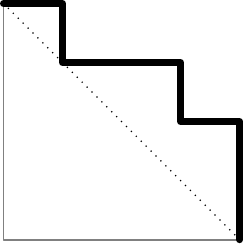}
    \caption{\label{fig:PF}At left, a column-restricted parking function in the set $\CPF(1,0,1,2)$ since the number of labels in each column from left to right are $1,0,1,2$.  At middle, a parking function which is not column-restricted, since $d_2=2$.  At right, their Dyck path, shown staying above the diagonal from $(0,4)$ to $(4,0)$.} 
\end{figure}

\subsection{Column-restricted parking functions}\label{sec:CPFs}

\begin{definition}\label{def:column-restricted}
Let $P$ be a parking function.  For every number label $x$ in $P$, we say $x$ \defn{dominates} a column to its right if the column contains no entry greater than $x$. Define the \defn{dominance index} $d_x$ to be the number of columns to the right of $x$ dominated by $x$ (including empty columns).  Then we say $P$ is \textbf{column-restricted} if $x>d_x$ for all $x=1,2,3,\ldots,n$.  We write $\CPF(k_1,\ldots,k_n)$ for the set of column-restricted parking functions having exactly $k_i$ labels in the $i$-th column for all $i$.
\end{definition}

\begin{thm}[Theorem 1.1 in \cite{CGM}]\label{thm:CPFs}
 We have $$\deg_{(k_1,\ldots,k_n)}(\emb_n)=|\CPF(k_1,\ldots,k_n)|.$$ 
\end{thm}

This theorem was proven by showing that the quantities $|\CPF(k_1,\ldots,k_n)|$ satisfy the recursion of Proposition \ref{prop:recursion}, using the following map.

\begin{definition}
We define $r : \PF(n) \to \PF(n-1)$ as follows: given a parking function $P$, $r(P)$ is defined by removing the row containing the number $1$, decrementing all remaining labels, and then deleting the rightmost empty column, as shown in Figure \ref{fig:r}. 
\end{definition}

\begin{figure}
    \centering
    \includegraphics{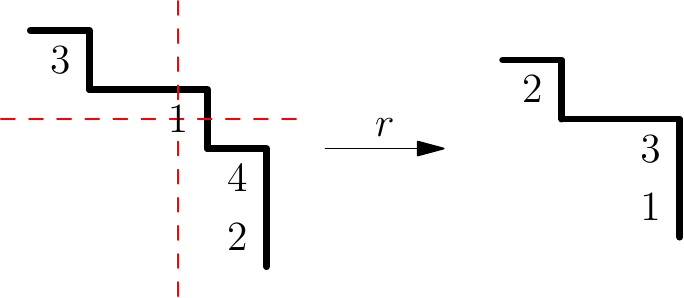} \hspace{2cm} \includegraphics{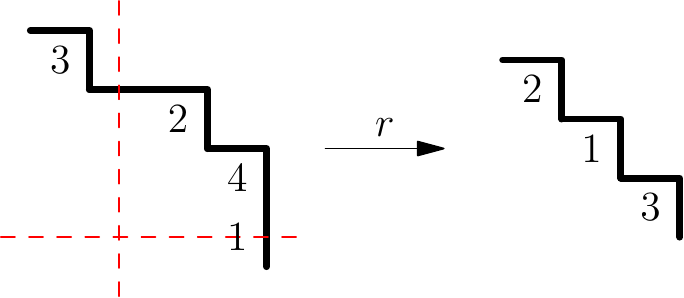}
    \caption{Applying the map $r$ to column-restricted parking functions.  At left is the case in which the $1$ is in its own column ($k_j=1$) and at right, the case in which $1$ shares a column with other labels ($k_j>1$).  Notice that the result is column-restricted in each case.}
    \label{fig:r}
\end{figure}

The following result shows that $r$ restricts to a bijection on column restricted parking functions.  (Note that in \cite{CGM}, $r$ was called $\varphi$.)

\begin{prop}[Proof of Theorem 5.3 in \cite{CGM}]
The map $r$ induces a bijection
\begin{equation}\label{eq:r-map}
\hat{r}:\CPF(\mathbf{k})\to \coprod_{j>i}\CPF(\widetilde{\bf k}_j).
\end{equation}  
\end{prop}

We will need the following additional lemma about CPFs.

\begin{lemma}\label{lem:reverse-cpf}
  Suppose $P \in \PF(n)$ has first dominance index $d_1=0$ and $r(P)$ is column-restricted.  Then $P$ is column-restricted as well.
\end{lemma}

\begin{proof}
  Let $i$ be the index of the rightmost empty column of $P$, and let $j$ be the column containing the $1$.  First suppose $x$ is a label in $P$ to the right of column $i$.  Then all columns to the right of $x$ are nonempty, and so $x$ dominates at most $x-1$ columns to its right (if these columns have largest entries $1,2,\ldots,x-1$).  Thus $d_x<x$ as required.
  
  Now suppose $x$ is a label to the left of column $i$ in $P$.  Then $x-1$ is the corresponding label in $r(P)$ and it dominates no more than $x-2$ columns to its right.  Then in $P$, $x$ dominates at most one more column to its right, namely either the empty column $i$ if $1$ is not in its own column, or column $j$ if $1$ is in its own column in column $j$.  Thus $d_x\le x-1$ as required. 
\end{proof}

\subsection{Bijection with tournaments}\label{sec:bijection}

We now construct an explicit bijection $\tau: \Tour(k_1,\ldots,k_n)\to \CPF(k_1,\ldots,k_n)$ that makes the following diagram of bijections commute, where $\hat{r}$ is the map \eqref{eq:r-map} defined above.
\begin{equation}\label{eq:diagram}
\begin{tikzcd}
\Tour(\mathbf{k}) \arrow[r,"\tau"] \arrow[d,"\forget"] & \CPF(\mathbf{k}) \arrow[d,"\hat{r}"] \\
\displaystyle\coprod_{j>i} \Tour(\widetilde{\bf k}_j) \arrow[r,"\coprod\tau"] & \displaystyle\coprod_{j>i} \CPF(\widetilde{\bf k}_j),
\end{tikzcd}
\end{equation}
Here  $\coprod\tau$ is induced by the maps $\tau : \Tour(\widetilde{\bf k}_j)\to \CPF(\widetilde{\bf k}_j)$.
This provides a direct link between the two combinatorial interpretations of the multidegrees. Recall our convention that the columns of a parking function are numbered $1,2,\dots, n$ from left to right.

\begin{definition}
  Given a trivalent tree $T\in \Tour(\mathbf{k})$, we define $\tau(T)$ to be the unique parking funtion of size $n$ such that for each $1\leq j,m\leq n$, the parking function $\tau(T)$ contains the number $m$ in column $j$ if and only if $j$ wins round $m$ of the tournament of $T$.
\end{definition}

\begin{example}
In the tournament $T$ in Example \ref{ex:tournament}, Round $1$ was won by the number $3$, so the parking function $\tau(T)$ has the number $1$ in column $3$.  Round $2$ was won by $4$, so the label $2$ appears in column $4$ in $\tau(T)$.  Round $3$ was won by $1$, so $3$ appears in column $1$, and Round $4$ was won by $4$, so $4$ appears in column $4$.  Thus $\tau(T)$ is the unique parking function whose sets of column labels, from left to right, are $\{3\},\{\},\{1\},\{2,4\}$, as shown in the upper right of Figure \ref{fig:commutative}.
\end{example}

\begin{figure}
    \centering
    \includegraphics{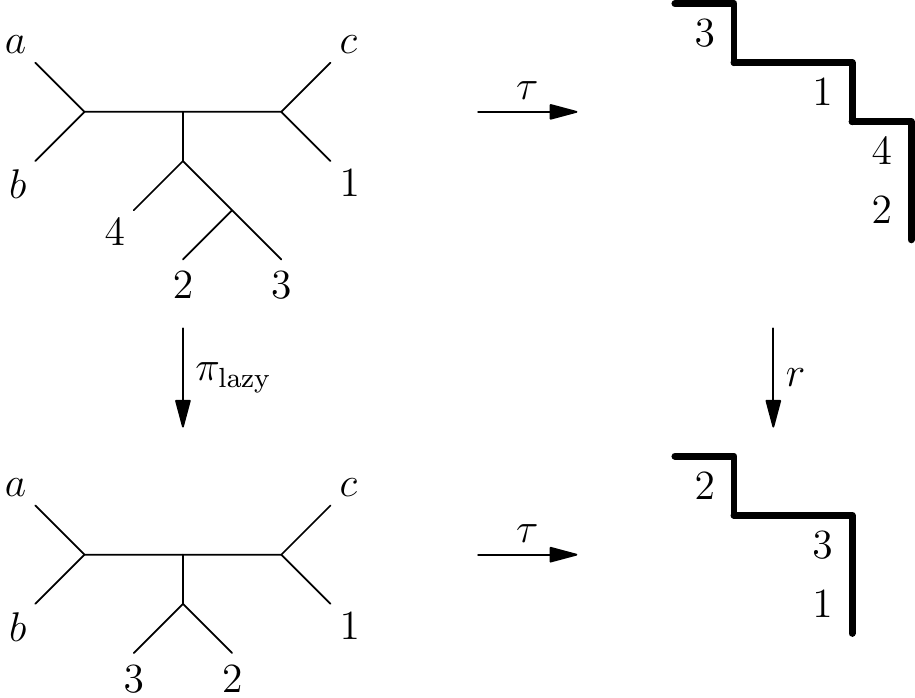}
    \caption{Tracing a tree $T\in \Tour(1,0,1,2)$ through the commutative diagram (\ref{eq:diagram2}).}
    \label{fig:commutative}
\end{figure}

As shown in \cite{CGM}, one can use the recursion of Proposition \ref{prop:recursion} to show that the multidegree $\deg_{(k_1,\ldots,k_n)}(\emb_n)$ is nonzero if and only if the sequence $(k_1,\ldots,k_n)$ is a reverse Catalan sequence, meaning that $k_{n}+k_{n-1}+k_{n-2}+\cdots+k_{n-i+1}>i$ for all $i$.  It is therefore an immediate consequence that applying $\tau$ to any tree $T\in \Tour(\mathbf{k})$ does indeed result in a parking function, with column heights $k_1,\ldots,k_n$.  

\begin{lemma}\label{lem:diagram-commutativity}
     For any $T\in \Tour(\mathbf{k})$, we have $r(\tau(T))=\tau(\forget(T))$.  That is, the following diagram commutes:
    \begin{equation}\label{eq:diagram2}
\begin{tikzcd}
\Tour(\mathbf{k}) \arrow[r,"\tau"] \arrow[d,"\forget"] & \PF(n) \arrow[d,"r"] \\
\displaystyle\coprod_{j>i} \Tour(\widetilde{\bf k}_j) \arrow[r,"\tau"] &  \PF(n-1).
\end{tikzcd}
\end{equation}
\end{lemma}

\begin{proof}

  Let $T\in \Tour(\mathbf{k})$, and let $i<j$ be the first pair that face off against each other in the tournament of $T$. Then $j$ is the column of the number $1$ in $\tau(T)$, and $i$ is the largest index for which $k_i=0$ by Lemma \ref{lem:biggest-loser}.  In particular, the $1$ is to the right of the rightmost empty column of $\tau(T)$.   The parking function $r(\tau(T))$ is the result of deleting the row containing the $1$ in $\tau(T)$, decrementing all remaining labels, and then deleting the rightmost empty column of the resulting diagram, which is column $i$ if $k_j>1$ and column $j$ otherwise.  This has the effect of decreasing the column indices of any label to the right of column $i$ or $j$ respectively by $1$.

  On the other hand, the tree $\forget(T)$ is formed by running the first round of the tournament, deleting the used leaf edges $i,j$, and decrementing all the labels above $j$ or $i$ respectively according to whether $k_j=1$ or $k_j>1$.  Under $\tau$, this corresponds to shifting all columns to the right of $j$ or $i$ respectively to the left one step, and also removing the $1$ and decrementing the remaining labels since the second round of the original tournament is now the first round of $\forget(T)$.  Thus $r(\tau(T))=\tau(\forget(T))$ as desired.
\end{proof}

We now need to show that the parking functions obtained from $\Tour(\mathbf{k})$ by applying $\tau$ are precisely the column-restricted parking functions.

\begin{prop} \label{prop:bij-tour-cpf}
The map $\tau$ is a bijection from $\Tour(\mathbf{k})\to \CPF(\mathbf{k})$ for any weak composition $\mathbf{k}=(k_1,\ldots,k_n)$ of $n$, and the diagram \eqref{eq:diagram} commutes.
\end{prop}

\begin{proof}
We first show that $\tau(T)$ is in $\CPF(\mathbf{k})$ by induction on $n$. In the base case $n=1$, this is easily checked, so assume the claim holds for compositions of $n-1$. Letting $i<j$ be the pair that faces off first in $T$, then $r(\tau(T)) = \tau(\forget(T))\in \CPF(\widetilde{\bf k}_j)$ by our inductive hypothesis. By the definition of $\tau$, the label $1$ is in column $j$ of $\tau(T)$. By Lemma~\ref{lem:biggest-loser}, $i$ is the largest index such that $k_i=0$, and since $i<j$, there are no empty columns to the right of the label $1$. Therefore, we have $1>d_1=0$ as in Definition~\ref{def:column-restricted} of column restrictedness. Since $r(\tau(T))$is also column restricted, we have that $\tau(T)$ is column restricted by Lemma~\ref{lem:reverse-cpf}. Hence, $\tau(T)\in \CPF(\mathbf{k})$ and the induction is complete.
  
  Because (\ref{eq:diagram2}) commutes by Lemma \ref{lem:diagram-commutativity} and $\tau(T)\in \CPF(\mathbf{k})$ for any $T\in \Tour(\mathbf{k})$, it follows that $(\ref{eq:diagram})$ commutes as well.  Finally, we show that $\tau : \Tour(\mathbf{k})\to \CPF(\mathbf{k})$ is a bijection by induction on $n$. The claim is easily checked for $n=1$, so assume the claim holds for all compositions of size $n-1$. The map $\forget$ is a bijection by Proposition~\ref{prop:bijection}, $\tau$ is a bijection by our inductive hypothesis, and $r$ is a bijection by~\cite{CGM}. Therefore, $\tau$ is also a bijection by commutativity of the diagram~\eqref{eq:diagram}.
\end{proof}

\subsection{Reversing the bijection}

We briefly describe, via the example below, how to compute $\tau^{-1}$ by hand, that is, how to recover the tree from its parking function (in a less cumbersome way than the insertion/relabeling argument of Proposition \ref{prop:bij-tour-cpf}).

\begin{example}
 Let $P$ be the parking function
 \begin{center}
 \includegraphics{PF-3--1-42.pdf}
 \end{center} from our previous examples.  The entries $1$ through $4$ in the diagram above are often called the \textbf{cars} of the parking function, and we refer to them as such below.  We compute $\tau^{-1}(P)$ as follows.
 
 \begin{enumerate}
     \item \textbf{List the winners and losers of the tournament.}  The indices of the columns of $P$ that contain cars are the winners, and the empty columns plus $c$ are the losers.  (In this example, $c$ and $2$ are the losers and $1,3,4$ are the winners.)  
     \item \textbf{Seed a rooted forest.}  For each $\ell\in \{c,1,2,\ldots,n\}$, draw a directed edge $\ell \rightarrow v_\ell$ labeled by $\ell$.  We consider each $v_\ell$ the root of its oriented tree.  
     \begin{center}
         \includegraphics[scale=0.8]{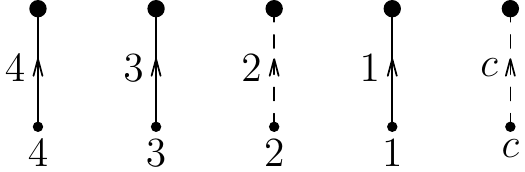}
     \end{center}
     For this example, we draw the loser edges with dashed lines and the winners with solid.  
     \item For each car $m$ of the parking function, starting from car $1$:
     \begin{itemize}
         \item \textbf{Merge two trees.}  Let $i$ be the largest loser that is adjacent to the root $v_i$ of its tree at this step.  If the car $m$ is in column $j$, let $v_j$ be the root of the tree containing $j$ (which in fact will be adjacent to an edge labeled $j$).   Identify vertex $v_i$ with $v_j$; $i$ faces off (and loses) against $j$.
         
         \item \textbf{Extend the tree.} Draw another directed edge $e$ from vertex $v_i=v_j$ to a new root $v_e$.  If car $m$ is at the top of column $j$ and it is not the final car $n$, then the laziness rule applies and we label $e$ by $i$; otherwise we label it by $j$. 
     \end{itemize}
         The first merge and extend steps for our example are shown below.
    \begin{center}
     \includegraphics[scale=0.8]{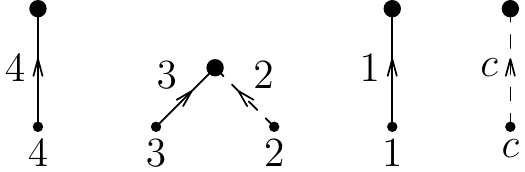} \hspace{4cm}
     \includegraphics[scale=0.8]{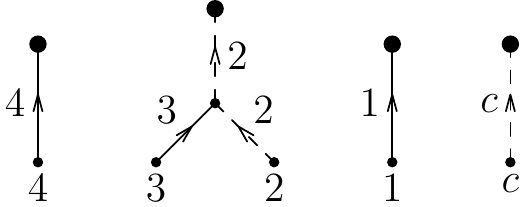}

    \end{center}
     We now repeat the merge and extend steps for cars $2,3,\ldots$.  The running example is shown below.
     \begin{center}
         \includegraphics[scale=0.8]{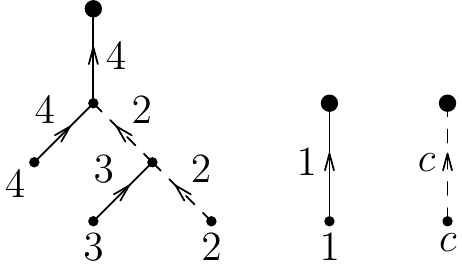} \hspace{2cm} \includegraphics[scale=0.8]{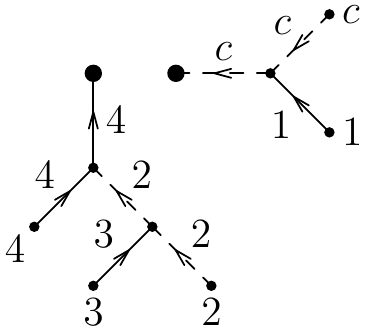}
         \hspace{2cm} \includegraphics[scale=0.8]{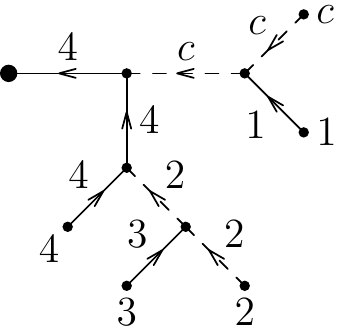}
     \end{center}

     \item \textbf{Add $a$ and $b$.}  When we have a single connected tree rooted at $v$, attach leaf edges $a$ and $b$ to $v$.
     \begin{center}
         \includegraphics{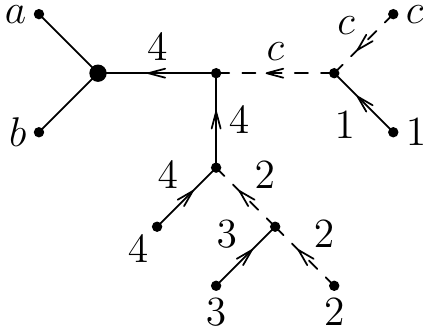}
     \end{center}
\end{enumerate}
Finally, removing the orientation and labels on each edge, we recover the tree $\tau^{-1}(P)$.
\end{example}
 
 \begin{remark}
 The tree can also be built from the parking function by starting with the vertex connected to $a,b$ and branching out, while reading the cars of the parking function from greatest to least (rather than least to greatest). In this case, the tournament is run in reverse, pairing the columns that the cars appear in with losers starting with the \textit{smallest} available loser. As in the method above, the key fact is that the laziness rule applies precisely when the corresponding car of the parking function is at the top of its column.
 \end{remark}

\section{Hyperplanes containing the tournament points}\label{sec:hyperplanes}

A natural question in light of Theorem \ref{thm:tournaments} is whether the set $\Tour(\mathbf{k})$ can be obtained as a complete intersection of $\Mbar_{0, X}$ with an appropriate set of hyperplanes in the iterated Kapranov embedding. As shown in Example \ref{ex:intro}, this is not possible in general, because the linear span of $\Tour(\mathbf{k})$ (in a given factor of the embedding) may intersect the image of $\Mbar_{0, X}$ in a subset of dimension $> \dim(\Mbar_{0,X}) - k_i$.

It is often possible, however, to express $\Tour(\mathbf{k})$ as the \emph{limit} of such an intersection, using a varying family of hyperplanes. We explore this question further in forthcoming work. Although we do not know if such a limit exists in general, a necessary condition is that there is a set of (fixed) linearly independent hyperplanes, $k_i$ of which lie in $\PP^i$ for each $i$, which contain every point of $\Tour(\mathbf{k})$.  We end by showing that such hyperplanes do, indeed, exist.  We restate Theorem \ref{thm:simple-hyperplanes} here for convenience.

\begin{HyperplanesThm}
Let $[z_b:z_c:z_1:z_2:\cdots:z_{r-1}]$ be the projective coordinates of the $\PP^r$ coordinate in $\PP^1\times \cdots \times \PP^n$.  Then the coordinates of the points of $\Tour(k_1,\ldots,k_n)$ in the $\PP^r$ factor all lie on the $k_r$ hyperplanes
$$z_b=0,\,\, z_c=0,\,\, z_1=0,\,\, \ldots,\,\, z_{k_r-2}=0,$$ where if $k_r=1$ then our collection only contains the hyperplane $z_b=0$, and if $k_r=2$ then we only have the two hyperplanes $z_b=0$ and $z_c=0$.  (If $k_r=0$ it is the empty collection.)
\end{HyperplanesThm}

To prove this, we first require two technical lemmas.  We will say that two labels $i,j$ are in different branches (resp., the same branch) \defn{from the perspective of $r$} in a tree $T$ if they are on different branches (resp., the same branch) from the internal vertex $v_r$ adjacent to leaf edge $r$.  If they are on different branches, we also say that $r$ \defn{separates} $i$ from $j$ in $T$.

\begin{lemma}\label{lem:separation}
Let $T\in \Tour(\mathbf{k})$, let $r$ be a winner in $T$, and suppose that the label $r$ in $T'=\pi_{r+1}\circ \cdots \circ \pi_{n}(T)$ separates some label $\ell$ from $a$. Then $r$ separates $\ell$ from $a$ in $T$, as well.
\end{lemma}

\begin{proof}
Note that $T$ is obtained from $T'$ by successively inserting the numbers $r+1,\ldots,n$ as leaf edges attached to existing edges starting from $T'$.  We claim that the property of $r$ separating $\ell$ from $a$ still holds in $T$.  Indeed, let $T'_r$, $T'_a$, and $T'_\ell$ be the three branches of the tree attached to $r$'s internal vertex in $T'$, which contain $r,a,\ell$ respectively.   If the labels $r+1,\ldots,n$ are all inserted at edges in either $T'_a$ or $T'_\ell$, it is clear that $r$ still separates $\ell$ from $a$.

If instead one of the labels $r+1,\dots, n$ is inserted on the unique edge in $T'_r$ (with possibly more inserted on the resulting edges), then $r$ would be paired in its first round of the tournament with some label among $r+1,\ldots,n$, and therefore $r$ loses its first round.  This is a contradiction to the Losers Lemma (\ref{lem:losers-lose}) since $r$ is a winner.  Thus $r$ separates $\ell$ from $a$ in $T$ as well.  
\end{proof}

\begin{lemma}\label{lem:branch-path}
Let $v$ be a vertex of a tree $T \in \Tour(\mathbf{k})$ and let $B$ be a branch at $v$ not containing $a$. Let $m$ be the smallest leaf label of $B$ and let $P$ be the path from $m$ to $v$. 

By the time $m$ first participates in a round of the tournament of $T$, every edge of $B$ is labeled except those along $P$.

Moreover, $m$ faces off against every labeled edge of $B$ attached to $P$ and advances until at least the vertex just before $v$ in $P$.
\end{lemma}

\begin{proof}
If $B$ consists only of the leaf edge $m$, the result holds trivially.  So assume $B$ contains at least one leaf besides $m$. 

Since $m$ is minimal in $B$, it is paired in its first round with another element $p>m$ in $B$, so $m$ is a loser of the tournament.  Since the sequence of losers weakly decreases (Lemma \ref{lem:losers-decrease}) and $m$ is minimal in $B$, all other pairs in $B$ will face off before $m$'s first round.  

Now, suppose for contradiction that some edge on path $P$ from $m$ to $v$ becomes labeled before $m$'s first round.  This is only possible if the two other edges adjacent to $v$ (not in branch $B$) are labeled and then face off to label an edge in path $P$.  However, by the Participation Lemma (\ref{lem:numbers-used}), labels that advance in the tournament do so along their path towards $a$, and since $a$ is not in branch $B$, we have a contradiction.  Hence $P$ is unlabeled until $m$ starts competing, at which point it advances by the laziness principle against all of its opponents in branch $B$ except possibly the last.
\end{proof}

\begin{proof}[Proof of Theorem \ref{thm:simple-hyperplanes}]
  First note that since the leaf edges of $a$ and $b$ share a vertex in all tournament points, $b$ is on $a$'s branch from the perspective of any other vertex of the tree, so $z_b=0$ always holds.  In particular, we only have to consider the case $k_r\ge 2$.
  
  Since $k_r\ge 2$, for any $T\in \Tour(k_1,\ldots,k_n)$, the number $r$ wins at least two rounds of the tournament of $T$ by the definition of $\Tour(\mathbf{k})$.  Let $P$ be the path from $c$ to $a$, let $v_a$ be the internal vertex at leaf edges $a,b$, and let $B$ be the branch from $v_a$ not containing $a,b$.  Then $c$ is the minimal label in $B$, so by Lemma \ref{lem:branch-path}, a leaf edge attached to $P$ only faces off against $c$ (and then $c$ advances by the laziness rule).  Thus $r$ itself is not directly attached to a vertex on path $P$.   Moreover, since $r$ can only face off once against $c$ if it advances to path $P$, $r$ wins against at least one other number $i<r$ in its branch off of $P$. In particular, in the tree $T' = \pi_{r+1} \circ \cdots \circ \pi_n(T)$, the leaf edge $r$ is still not attached to path $P$.  Thus, in $T'$, leaves $a$ and $c$ are on the same branch from the perspective of $r$, so $z_c=0$ by Corollary \ref{cor:full-coordinates}.
  
  We now show that if $k_r>2$, the coordinates of the point $T$ satisfy the additional equations $z_1=0,z_2=0,\ldots,z_{k_r-2}=0$.  Assume for contradiction that $z_\ell = 1$ for some $\ell\le k_r-2$.  By Corollary \ref{cor:full-coordinates}, this means that in $T'$, the label $r$ separates $\ell$ from $a$. By Lemma \ref{lem:separation}, $r$ separates $\ell$ from $a$ in $T$ as well. Now, let $v_r$ be the internal vertex adjacent to $r$, let $T_\ell$ be the branch from $v_r$ containing $\ell$, and let $m$ be the smallest label $T_\ell$. By Lemma \ref{lem:branch-path}, since $T_\ell$ does not contain $a$, we have that $m$ labels all edges in its path to $v_r$ except possibly the last edge (connecting to $v_r$). However, note that 
  \[m\le\ell\le k_r-2\le r-2 < r,\]
  so $m$ also advances to the final edge adjacent to $v_r$ by the laziness principle.
  
  It follows that $r$'s first round of the tournament is against some number $m\le k_r-2$.  By the Winners and Losers Lemmas (\ref{lem:winners-win} and \ref{lem:losers-lose}), $r$ wins every round in which it competes. By Lemma \ref{lem:losers-decrease}, the losers (across the entire tournament) form a weakly decreasing sequence. Furthermore, $r$ itself will never face the same opponent twice, and so the sequence of losers that $r$ faces form a \textit{strictly} decreasing sequence starting at $m$. Thus by the Participation Lemma (\ref{lem:numbers-used}), the maximum possible number of opponents $r$ has is $m+1$ (since $c,1,2,\ldots,m$ may be its opponents, but not $a$ or $b$).  But $m+1\le k_r-1$, and so $r$ wins at most $k_r-1$ times, contradicting the fact that $T \in \Tour(\mathbf{k})$.
  
  Hence $z_\ell=0$ as desired.
\end{proof}

\bibliography{myrefs}

\begin{thebibliography}{10}

\bibitem{BergstromMinabe}
Jonas Bergström and Satoshi Minabe.
\newblock On the cohomology of moduli spaces of (weighted) stable rational
  curves.
\newblock {\em Mathematische Zeitschrift}, 275(3–4):1095–1108, 2013.

\bibitem{Blasiaketc}
Jonah Blasiak, Jennifer Morse, Anna Pun, and Daniel Summers.
\newblock Catalan functions and k-{S}chur positivity.
\newblock {\em J. Amer. Math. Soc.}, 32:921--963, 2019.

\bibitem{canning2021chow}
Samir Canning and Hannah Larson.
\newblock The {C}how rings of the moduli spaces of curves of genus 7, 8, and 9.
\newblock 2021.
\newblock arXiv:2104.05820.

\bibitem{CGM}
Renzo Cavalieri, Maria Gillespie, and Leonid Monin.
\newblock Projective embeddings of $\overline{M}_{0,n}$ and parking functions.
\newblock {\em Journal of Combinatorial Theory, Series A}, 182:105471, 2021.

\bibitem{chan-pflueger}
Melody Chan and Nathan Pflueger.
\newblock Euler characteristics of {B}rill-{N}oether varieties.
\newblock {\em Trans. Amer. Math. Soc.}, 374(3):1513--1533, 2021.

\bibitem{clader2021permutohedral}
Emily Clader, Chiara Damiolini, Daoji Huang, Shiyue Li, and Rohini Ramadas.
\newblock Permutohedral complexes and rational curves with cyclic action.
\newblock 2021.
\newblock arXiv:2104.06526.

\bibitem{CladerJanda}
Emily Clader and Felix Janda.
\newblock {Pixton's double ramification cycle relations}.
\newblock {\em Geometry and Topology}, 22(2):1069 -- 1108, 2018.

\bibitem{clader2020boundary}
Emily Clader, Dante Luber, and Kyla Quillin.
\newblock Boundary complexes of moduli spaces of curves in higher genus.
\newblock 2020.
\newblock arXiv:2007.09710.

\bibitem{damiolini2020vertex}
Chiara Damiolini, Angela Gibney, and Nicola Tarasca.
\newblock Vertex algebras of {C}oh{F}{T}-type.
\newblock 2020.
\newblock arXiv:1910.01658.

\bibitem{deligne-mumford1969}
Pierre Deligne and David Mumford.
\newblock The irreducibility of the space of curves of given genus.
\newblock {\em Inst. Hautes \'Etudes Sci. Publ. Math.}, 36:75--109, 1969.

\bibitem{eisenbud-harris-limit-linear}
David Eisenbud and Joe Harris.
\newblock Limit linear series: basic theory.
\newblock {\em Invent. Math.}, 85(2):337--371, 1986.

\bibitem{3264}
David Eisenbud and Joe Harris.
\newblock {\em 3264 and All That: A second course in algebraic geometry}.
\newblock Cambridge University Press, 2016.

\bibitem{fry2019tropical}
Andy Fry.
\newblock Tropical moduli space of rational graphically stable curves.
\newblock 2019.
\newblock arXiv:1910.00627.

\bibitem{Getzler}
E.\ Getzler.
\newblock Operads and moduli spaces of genus 0 {R}iemann surfaces.
\newblock In G.B.M. van der~Geer R.H.~Dijkgraaf, C.F.~Faber, editor, {\em The
  Moduli Space of Curves.}, volume 129 of {\em Progress in Mathematics}.
  Birkh\"{a}user Boston.

\bibitem{gibneykeelmorrison2002}
Angela Gibney, Sean Keel, and Ian Morrison.
\newblock Towards the ample cone of {$\Mbar_{g,n}$}.
\newblock {\em J. Amer. Math. Soc.}, 15(2):273--294, 2002.

\bibitem{HaglundBook}
J.~Haglund.
\newblock {\em The $q,t$-Catalan Numbers and the Space of Diagonal Harmonics},
  volume~10 of {\em University Lecture Series}.
\newblock Amer.\ Math Soc., 1993.

\bibitem{Ka1}
Mikhail~M Kapranov.
\newblock Veronese curves and {G}rothendieck-{K}nudsen moduli space {$\Mbar_{0,
  n}$}.
\newblock {\em J. Algebraic Geom}, 2(2):239--262, 1993.

\bibitem{KeT}
Sean Keel and Jenia Tevelev.
\newblock Equations for {$\Mbar_{0, n}$}.
\newblock {\em Int. J. Math.}, 20(09):1159--1184, 2009.

\bibitem{larson2020global}
Eric Larson, Hannah Larson, and Isabel Vogt.
\newblock Global {B}rill--{N}oether {T}heory over the {H}urwitz {S}pace.
\newblock 2020.
\newblock arXiv:2009.10765.

\bibitem{MonRan}
Leonid Monin and Julie Rana.
\newblock Equations of {$\overline {M}_{0,n}$}.
\newblock In {\em Combinatorial algebraic geometry}, volume~80 of {\em Fields
  Inst. Commun.}, pages 113--132. Fields Inst. Res. Math. Sci., Toronto, ON,
  2017.

\bibitem{pandharipande2020relations}
R.~Pandharipande and A.~Pixton.
\newblock Relations in the tautological ring of the moduli space of curves.
\newblock 2020.
\newblock arXiv:1301.4561.

\bibitem{PixtonThesis}
Aaron Pixton.
\newblock {\em The tautological ring of the moduli space of curves}.
\newblock PhD thesis, Princeton University, Princeton, NJ, 2013.
\newblock URI: http://arks.princeton.edu/ark:/88435/dsp01t722h888k.

\bibitem{Rains}
Eric Rains.
\newblock The action of {$S_n$} on the cohomology of
  $\overline{M}_{0,n}(\mathbb{R})$.
\newblock {\em Selecta Mathematica - New Series}, 15, 01 2006.

\bibitem{RaSil2020}
Rohini Ramadas and Rob Silversmith.
\newblock Two-dimensional cycle classes on $\overline{M}_{0,n}$.
\newblock 2020.
\newblock arXiv:2004.05491.

\bibitem{sharma2019intersections}
Nand Sharma.
\newblock Psi-class intersections on {H}assett spaces for genus 0 with all
  weights {$\frac {1}{2}$}.
\newblock {\em Rocky Mountain J. Math.}, 49(7):2297--2324, 2019.

\bibitem{silversmith2021crossratio}
Rob Silversmith.
\newblock Cross-ratio degrees and perfect matchings.
\newblock 2021.
\newblock arXiv:2107.04572.

\bibitem{van}
B.~L.~Van~Der Waerden.
\newblock On varieties in multiple-projective spaces.
\newblock {\em Indagationes Mathematicae (Proceedings)}, 81(1):303--312, 1978.

\end{thebibliography}
\bibliographystyle{plain}

\end{document}